\documentclass[12pt]{amsart}
\usepackage[utf8]{inputenc}

\usepackage{amssymb}
\usepackage{epsfig}
\usepackage{amsfonts}
\usepackage{amsmath}
\usepackage{euscript}
\usepackage{amscd}
\usepackage{amsthm}

\DeclareMathAlphabet{\mathpzc}{OT1}{pzc}{m}{it}
\usepackage{mathrsfs}

\usepackage{enumitem}

\DeclareSymbolFont{SY}{U}{psy}{m}{n}
\usepackage{tikz-cd}

\usepackage{ wasysym }

\usepackage{stmaryrd}

\theoremstyle{plain}

\newtheorem{thm}{Theorem}[section]
\newtheorem{cor}[thm]{Corollary}
\newtheorem{lem}[thm]{Lemma}
\newtheorem{prop}[thm]{Proposition}
\theoremstyle{definition}
\newtheorem{defn}[thm]{Definition}
\newtheorem{rem}[thm]{Remark}

\numberwithin{equation}{section}
\def\norm#1{\left\lVert{#1}\right\rVert}
\def\setof#1{\{{#1}\}}

\def\inp#1,#2{\left\langle{#1},{#2}\right\rangle}
\def\N{\mathcal{N}}
\def\bN{\mathbb{N}}

\def\R{\mathbb{R}}

\def\m{\mathcal}

\def\P{\mathcal{P}}
\def\e{\varepsilon}

\def\beq{\begin{eqnarray}}
\def\eeq{\end{eqnarray}}
\def\beqa{\begin{eqnarray*}}
	\def\eeqa{\end{eqnarray*}}

\def\M{\mathcal{M}}

\def\<{\langle}
\def\>{\rangle}
\def\Z{\mathbb{Z}}

\def\c*{C^{*}}
\def\*{*}

\def\eps{\varepsilon}
\def\P{\mathcal{P}}

\def\m*{m^\*}

\def\F{\mathcal{F}}

\def\actson{\curvearrowright}

\def\P{\mathscr{P}}

\def\supp{\textrm{supp}}

\def\c{\mathcal{C}}

\def\interval#1{\left\llbracket{#1}\right\rrbracket}
\def\abs#1{\left\lvert{#1}\right\rvert}

\newcounter{cnt1}
\newcounter{cnt2}
\newcounter{cnt3}
\newcommand{\blr}{\begin{list}{$($\roman{cnt1}$)$}
		{\usecounter{cnt1} \setlength{\topsep}{0pt}
			\setlength{\itemsep}{0pt}}}
	\newcommand{\bla}{\begin{list}{$($\alph{cnt2}$)$}
			{\usecounter{cnt2} \setlength{\topsep}{0pt}
				\setlength{\itemsep}{0pt}}}
		\newcommand{\bln}{\begin{list}{$($\arabic{cnt3}$)$}
				{\usecounter{cnt3} \setlength{\topsep}{0pt}
					\setlength{\itemsep}{0pt}}}
			\newcommand{\el}{\end{list}}

\usepackage{geometry}
\usepackage{hyperref}
\hypersetup{
    colorlinks=true,
    linkcolor=blue,
    filecolor=blue, 
    citecolor=blue,
    urlcolor=red
    }

\allowdisplaybreaks[1]

\title[Ergodic Average Dominance]{Ergodic Average Dominance for Unimodular Amenable Groups}

\author[Ujan Chakraborty]{Ujan Chakraborty%\textsuperscript{1} \orcidlink{0000-0002-2655-4488}
}
\address{School of Mathematics and Statistics, University of Glasgow, University Avenue, Glasgow G12 8QQ, UK}
\email[Ujan Chakraborty]{Ujan.Chakraborty@glasgow.ac.uk}

\author[Runlian Xia]{Runlian Xia%\textsuperscript{2} \orcidlink{0000-0002-0516-3701}
}
\address{School of Mathematics and Statistics, University of Glasgow, University Avenue, Glasgow G12 8QQ, UK}
\email[Runlian Xia]{Runlian.Xia@glasgow.ac.uk}

\author[Joachim Zacharias]{Joachim Zacharias%\textsuperscript{3} \orcidlink{0000-0003-1186-3816}
}
\address{School of Mathematics and Statistics, University of Glasgow, University Avenue, Glasgow G12 8QQ, UK}
\email[Joachim Zacharias]{Joachim.Zacharias@glasgow.ac.uk}



\begin{document}

\begin{abstract}
We prove that ergodic averages of actions of second countable unimodular amenable groups along suitable F{\o}lner sequences are dominated by Ces\`aro means of a certain Markov operator, i.e. by ergodic averages of an integer action. The required condition on the F{\o}lner sequence is mild: every two-sided F{\o}lner sequence admits a subsequence satisfying it. %, and the construction of the Markov operator depends only on the averaging F{\o}lner sequence. 
As a consequence, maximal and pointwise ergodic theorems for unimodular amenable group actions follow directly from the classical results for integer actions. Whilst other proofs of these ergodic theorems using various harmonic analytic techniques are known, our combinatorial method is considerably simpler and applies to both commutative and noncommutative settings, putting them on an equal footing and avoiding technical difficulties associated with group complexity or noncommutativity. The domination principle introduced here may be of independent interest in ergodic theory.

    %We show that the ergodic averages of the action of any separable unimodular amenable group along certain F{\o}lner sequences can be dominated by the Ces\`aro means of
    %a suitably constructed Markov operator, that is, the ergodic averages of an integer action. 
    %Moreover, the restriction on these F{\o}lner sequences is mild enough so that every two-sided F{\o}lner sequence has a subsequence satisfying these conditions. 
    %As a consequence, %of this inequality, 
    %we derive the  maximal and pointwise (individual) ergodic theorems for actions of unimodular amenable groups directly from the corresponding ergodic theorems for integer actions. 
    %Although proofs of such ergodic theorems for general amenable groups already exist in the literature, our approach is considerably simpler than existing methods. Moreover our ergodic average dominance approach allows us to treat the commutative and noncommutative settings on an equal footing circumventing difficulties proving maximal inequalities which arise from the complexity of the group or the noncommutativity of the underlying space. This dominance may also have potential applications to other problems in ergodic theory.
    
    %This average dominance is a direct consequence of an inequality concerning comparison of probability measures on the group: or in other words, we show that this class of problems in ergodic theory can be addressed as a class of problems in geometric group theory.
\end{abstract}

\maketitle

\section{Introduction}
    Ergodic theory has been an area of active research for more than a century, since Boltzmann proposed the Ergodic Hypothesis for molecular dynamics in 1898. %\cite{Bo1898}.
    The rigorous mathematical formulation of these hypotheses took decades and was initially pioneered by the likes of von Neumann \cite{vN32} and Birkhoff \cite{Bi31}, who proved ergodic theorems for flows (action of $\mathbb{R}$).  
    When $T$ is an ergodic measure preserving transformation of a probability space $(X,\mu)$, Birkhoff's result implies that for integrable functions $f$, $$M_n f(s):=\frac{1}{n}\sum_{j=0}^{n-1}f(T^js) \longrightarrow\int_X f d\mu, 
    \text{ a.e. } s\in X.$$
Von Neumann's slightly earlier result had already shown convergence in the $L_2$-norm.
    %It was known already that convergence in norm also holds, that is, $\norm{\frac{1}{n}\sum_{j=0}^{n-1}T^jf - \int f d\mu}_1 \longrightarrow 0$, and the proof of this fact is attributed to von Neumann. He proved a much more general statement \cite{vN32}, in fact, for actions of $\mathbb{R}$ on Hilbert spaces through unitaries. 
    %For square-integrable functions, this reduces to convergence in $L_2$-norm. %(also for isometric transformations on more general Hilbert spaces) was proved by von Neumann \cite{vN32}: that is, $\norm{\frac{1}{n}\sum_{j=0}^{n-1}T^jf - \int f d\mu}_2 \longrightarrow 0$. 
    %The proof of the mean ergodic theorem for Hilbert spaces was given by von Neumann \cite{vN32}, while the pointwise ergodic theorem for measure spaces was first proven by Birkhoff \cite{Bi31}. 
   % Mean ergodic theorems (convergence in norm) could also be proved in a fairly straightforward fashion on $L_p$-spaces, and reflexive Banach spaces in general.
    %Pointwise almost everywhere convergence were proved as well, for $L_p$-spaces with $p\in[1,\infty)$, and these results were extended from probability spaces to general $\sigma$-finite measure spaces. 
    The condition of ergodicity on $T$ can also be relaxed by replacing the limit by a projection $\mathcal{P}$ onto the subspace of $T$-invariant functions.
    %But this was all in the context of actions by $\mathbb{Z}$ (acting by the powers of the same map $T$) or $\R$ (acting through a one-parameter family of maps $T_t$). 
    %Techniques to prove ergodic theorems were refined and honed over the subsequent decades and extended to larger classes of groups. 
    With the evolution of dynamics through the decades, ergodic theory came to the study of the convergence of ergodic averages for more general groups.  
    One natural class of groups to consider is the class of amenable groups. % \cite{vN29}, also introduced by von Neumann. 
    For such groups, there is a natural way to take ergodic averages using F{\o}lner sequences, but this process becomes very delicate in this general setting, 
    so that establishing pointwise ergodic theorems for general locally compact amenable groups was an open problem for a long time. It was only fully resolved in 2001 by Lindenstrauss who provided the first proof of the following theorem in \cite{Li01}.
 Let $G$ be a second countable amenable group with Haar measure $\lambda$. 
    A F{\o}lner sequence for $G$ is said to be tempered if there exists a constant $C>0$ such that $\lambda\left(\bigcup_{i=1}^{n-1}F_i^{-1} F_n \right) \leq C \lambda(F_n)$ for all $n \geq 2$. Any F{\o}lner sequence admits a subsequence satisfying this temperedness condition. 
\begin{thm}[Lindenstrauss (2001) \cite{Li01}]\label{thm: Lindenstrauss} Let $G$ be a locally compact and second countable amenable group, $(F_n)$ a tempered F{\o}lner sequence and 
$\alpha$ a measure preserving action of $G$ on a probability space $(X,\mu)$ then 
$$A_n(f)(s):=\frac{1}{\lambda(F_n)}\int_{F_n}\alpha_g(f) (s) d\lambda(g) \longrightarrow \mathcal{P}(f)(s), \text{ a.e. } s\in X,$$
for any $f\in L_p(X,\mu)$, $1\leq p<\infty$,
    where $\mathcal{P}$ is the projection onto the subspace of $G$-invariant functions.
\end{thm} 
    We outline a standard recipe to prove such pointwise convergence below, which was also followed in \cite{Li01}. First, norm convergence in $L_p$ is obtained fairly easily, and this implies almost everywhere convergence for all bounded functions in $L_p$, which form a dense subset. Then, overcoming considerable technical difficulties, one establishes a maximal inequality via transference principles \cite{CW76}. Finally, one uses the Banach principle to extend the property of almost everywhere convergence from the dense subset of bounded functions in $L_p$ to all of $L_p$. 
   The maximal inequalities are not easy to obtain for general amenable groups, and usually constitute the most involved step of the process. 
   The reader is directed to \cite{AAB10} and \cite{Ne06} for further details. 
    
    In parallel, von Neumann algebras %\cite{MvN36, MvN37, MvN43} 
    came to the forefront of contemporary analysis as ``noncommutative measure spaces'' and questions arose regarding the possibility of establishing ergodic theorems in the setting of von Neumann algebras.
%Initially only mean ergodic theorems were proved \cite{Ja85, Ja91}, while the question of what ``pointwise'' convergence might even mean in this context  remained unclear for a while. 
    The first question is what pointwise convergence might mean in this context. For finite measure spaces, restricting to bounded functions, pointwise almost everywhere convergence implies almost uniform convergence, by Egorov's theorem, meaning that for any $\varepsilon>0$, there exists $X_\varepsilon \subset X$ such that $\mu(X\backslash X_\varepsilon)<\varepsilon$, and $\lim_{n\rightarrow \infty}\norm{\left(M_n(f)-\mathcal{P}(f)\right)\chi_{X_\varepsilon}}_\infty =0$. 
    In this formulation one can see that this notion of convergence  makes complete sense in the noncommutative setting as well, if one replaces $f$ with von Neumann algebra elements (or affiliated operators), %$T$ with an automorphism (or, as we shall later see, even positive maps which are contractions in norm and trace. 
     and $\chi_{X_\varepsilon}$ with an in a suitable sense ``large" projection $e$ in the von Neumann algebra, leading to the notion of almost uniform convergence (Definition \ref{defintion au bau}) in the noncommutative setting. 
     %In the case of noncommutative von Neumann algebras, for a projection $e$, the cases 
     In this setting we may require $\lim_n \norm{(x_n-x)e}_\infty = 0$ %(almost uniform) 
     or $\lim_n \norm{e(x_n-x)e}_\infty =0$, defining almost uniform (a.u.) or bilateral almost uniform (b.a.u.) convergence, respectively. These concepts are distinct, the former implying the latter but not conversely.  
    We shall refer to ergodic theorems with this notion of ``pointwise'' convergence as ``individual'' ergodic theorems in our paper. 
    The first individual ergodic theorem in this setting was proved by Lance \cite{La76} for state-preserving automorphisms of finite von Neumann algebras. 
Then Yeadon \cite{Ye77} proved ergodic theorems for preduals of semifinite von Neumann algebras (the case for $L_1$). However, ergodic theorems  for general noncommutative $L_p$-spaces for $p\in(1, \infty)$ were more involved mainly due to difficulties in formulating the maximal inequalities in the noncommutative case. This was resolved by Junge and Xu in \cite{JX07}.  %They showed that if $L_p(\M)$ is a noncommutative $L_p$-space with a positive linear $T$ that is a contraction in norm and in trace, then one has a maximal inequality of the following form: 
   % for any positive $x \in L_p(\M)$, there exists a positive $a \in L_p(\M)$ satisfying $M_n(x)\leq a$ for all $n \in \mathbb{N}$ and $\norm{a}_p \leq C_p \norm{x}_p$ where $C_p$ is a constant depending only on $p$. This allows one to obtain individual ergodic theorems. 
    Later, Hong, Liao and Wang \cite{HLW21} extended their results to groups of polynomial growth. 
    Very recently, Cadilhac and Wang achieved a major breakthrough in \cite{CW22}, proving the following 
    individual ergodic theorems for amenable groups acting on noncommutative $L_p$-spaces.
    
\begin{thm}[Cadilhac and Wang (2022)]\label{thm: Cailhac and Wang}
Let $G$ be amenable and $\alpha$ be a trace preserving action on a tracial von Neumann algebra $\M$. Then there is a F{\o}lner sequence $(F_n)$ such that for any $x \in L_p(\M)$ with $1 \leq p <2$ (resp. $2\leq p < \infty$),
$$A_n(x):=\frac{1}{\lambda(F_n)}\int_{F_n}\alpha_g(x) d \lambda(g) \to  P(x)$$
bilaterally almost uniformly (resp. almost uniformly),
where $P$ is the projection onto the space of fixed points of $\alpha$.
\end{thm}    
    %that for any amenable group $G$ acting on a semifinite von Neumann algebra $(\M, \tau)$ by an w* continuous $\tau$-preserving action $\alpha$, there exists a ``filtered'' F{\o}lner sequence $\{ F_n \}_{n \in \mathbb{N}}$ %(which is possible to choose as a subsequence of any F{\o}lner sequence) 
    %such that defining the averaging operator $A_n:L_p(\M) \to L_p(\M)$ by $x \to \frac{1}{\lambda(F_n)}\int_{F_n}\alpha_g(x) d \lambda(g)$,
    %one has $A_n(x) \longrightarrow P(x)$, where $P$ is the projection onto the subspace of invariant elements, and the convergence is b.a.u. for $p\in[1,2)$ and a.u. for $p\in[2, \infty)$. 
    Their proof of the maximal inequality involves noncommutative Calder{\'o}n–Zygmund decomposition and transference techniques, dyadic-like martingale methods on amenable groups,
 and quasi-tiling techniques to construct admissible F{\o}lner sequences. 

%\begin{nonumberthm}
%Test
%\end{nonumberthm}

 The goal of our paper is to prove an ergodic average dominance (see Theorem 1.3 or \ref{action dominance}) for unimodular amenable groups, using random walk techniques. As an application of this, we provide a new proof of individual ergodic theorems (both commutative and noncommutative) for these groups, circumventing the difficulties associated with proving the maximal inequalities arising from either the amenable group structure or the noncommutativity of the underlying space.
 There is a long history of using random walk techniques to prove ergodic theorems: 
    in particular, Oseledets \cite{Os65} proved weighted ergodic theorems for the iterates of Markov operators associated to actions of arbitrary locally compact groups, so effectively integer actions but without relating them to any averages over the group. %using %random walks and 
    %Markov operators. 
    There is also a rich and longstanding tradition of studying random walks on groups, particularly on amenable groups. %starting from Kesten \cite{Ke59a}, \cite{Ke59b}, which has continued throughout in several directions. 
    One particular set of results was provided by Kaimanovich and Vershik \cite{KV83}, which demonstrated that for any discrete countable amenable group, there exists a nondegenerate probability measure (i.e. the semigroup generated by its support is the whole group) whose convolution powers yield a Reiter sequence. However, to the best of our knowledge, this result has not been applied to prove ergodic theorems associated with F{\o}lner sequences for amenable groups. 
    Inspired by \cite{KV83},  
    we show in Proposition \ref{measure comparison in terms of sizes} that for any unimodular locally compact second countable amenable group $G$, a nondegenerate probability measure $\omega$ can be constructed such that the Ces\`aro means of its convolution powers
    dominate (up to a constant $C$) the uniform probability measures supported on a F{\o}lner sequence $\{F_n\}_{n \in \mathbb{N}}$:
    \begin{align*}\label{MC}
    {\frac{\chi_{F_n}}{\lambda(F_n)} \leq C \frac{1}{N(n)}\sum_{j=0}^{N(n)-1}\frac{d\omega^{(j)}}{d\lambda},\tag{MC}}
    \end{align*}
    where $N$ is an increasing function on $\mathbb{N}$. 
    This probability measure $\omega$ then provides us with the dominating Markov operator $T(x)=\int_G \alpha_g(x) d\omega(g)$, where $x\in L_p(\M)$ and  $\omega \in L_1^+(G, \lambda)$ only depending on $\{ F_n \}_{n \in \mathbb{N}}$, not on $\alpha$, allowing us to prove our main result in section 3:
    \begin{thm}(cf \ref{action dominance} below)
    For every two-sided F{\o}lner sequence of $G$, there exists a subsequence %of compact subsets of $G$ 
    $\{F_n\}_{n \in \mathbb{N}}$, a constant $C > 0$, a strictly increasing function $N:\mathbb{N} \to \mathbb{N}$, and 
    %a  $T$ satisfying conditions \ref{Tprops1}, \ref{Tprops2}, \ref{Tprops3} 
    a completely positive contractive trace preserving linear map $T$ 
    %\ref{Tprops1}, \ref{Tprops2}, \ref{Tprops3}
    such that for any $x$ in $L_p^+(\M)$, {$p \in [1, \infty]$} and %any 
    large enough $n$,
    \begin{align}\label{Ergodic Operator Inequality}
    \frac{1}{\lambda{(F_n)}} \int_{F_n} \alpha_g(x) d\lambda(g) &\leq C \frac{1}{N(n)}\sum_{j=0}^{N(n)-1} T^j(x).\tag{AD}
    \end{align} 
    %Moreover, $T$ can be realised as a Markov operator $T(x)= \int_G \alpha_g(x) d\omega(g)$ for some probability measure $\omega$ on $G$
    %absolutely continuous  
    %with respect to $\lambda$, i.e.
    %$\omega \in L_1^+(G, \lambda)$, and $\omega$ depends only on $\{ F_n \}_{n \in \mathbb{N}}$.
\end{thm}
    %If we consider a {$w^\star$-continuous} trace-preserving action $\alpha$ of $G$ on a von Neumann algebra $\M$ (when $\M$ is commutative, we are back to the setting of measure spaces), this allows us to dominate the ergodic averages along that F{\o}lner sequence by the ergodic averages of the Markov operator corresponding to the nondegenerate probability measure, defined as $T(x)=\int_G \alpha_g(x) d\omega(g)$ for any $x\in L_p(\M)$ with $p \in [1, \infty]$:
    %$$\frac{1}{\lambda{(F_n)}} \int_{F_n} \alpha_g(x) d\lambda(g) \leq C \frac{1}{N(n)}\sum_{j=0}^{N(n)-1} T^j(x).$$ 
    %This ergodic average dominance is the main result of our paper. 
    As applications, we obtain Theorem \ref{thm: Lindenstrauss} and Theorem \ref{thm: Cailhac and Wang} for unimodular amenable groups acting on tracial and nontracial von Neumann algebras
    directly from the corresponding theorems for integer actions in \cite{JX07}.
    This technique of proving ergodic theorems for unimodular amenable groups is considerably simpler than the ones in \cite{CW22} and \cite{Li01},
    since we no longer need to prove maximal and individual ergodic theorems from scratch using heavy machinery. 
    The traditional way to prove a maximal inequality for amenable groups proceeds via a transference principle in the same lines as Calder{\`o}n \cite{Ca68, CW76}, which shows that the maximal inequality for the action of the group on a measure space / a von Neumann algebra can be inferred from a maximal inequality for the translation action of the group on itself. 
    Instead, in our approach, easier functional-analytic techniques allow us to reduce the case of a group action to the integer action case using \eqref{Ergodic Operator Inequality}, and further reduce the problem to a comparison of probability measures on the group \eqref{MC}. 
    Our innovative technique essentially transfers the difficulty of proving a maximal inequality for the group action (a problem in noncommutative harmonic analysis and ergodic theory) to the construction of an appropriate probability measure on the group by combinatorial techniques (a problem in geometric group theory). 
    %This new recipe of proving ergodic theorems substitutes key harmonic-analytic ingredients like transference at the level of the group action by combinatorial results like measure comparison at the level of the group. 
    The core techniques being purely combinatorial, do not discriminate between the classical and noncommutative cases.
    %Our method works only in the unimodular case, and we do not yet know how to generalise this further. 
    
    We should also point out that the idea of comparing measures using random walk techniques was initiated in \cite{HLW21} to prove ergodic theorems for groups of polynomial growth. 
    Similar Markov operator techniques have been used in \cite{Buf02, ADe06} to prove ergodic theorems for free groups. However, the Gaussian estimates used in \cite{HLW21} to obtain the comparison cannot be extended beyond groups of polynomial growth and appear to require highly non-trivial new ideas to work for general amenable groups.
    %and in \cite{DN19} to prove similar theorems for generalisations of Kazhdan's property (T).
    %At present we do not know how to generalise our methods to the nonunimodular case, we have some potential obstructions to such results, detailed in \cref{potential obstructions for nonunimodular groups}. 

    The structure of the paper is as follows: 
    in Section 2, we provide the preliminaries for our paper, in particular regarding noncommutative $L_p$-spaces, some background in noncommutative harmonic analysis, and the Junge-Xu ergodic theorem. 
    In section 3, we prove our main results: we formulate the ergodic average dominance stated above and prove it through comparison of probability measures on the group.  
    In section 4, we show the maximal inequality as a corollary, and we also provide  proofs of individual ergodic theorems directly from our ergodic average dominance. 
    In section 5, we extend our results to the nontracial setting (Haagerup $L_p$-spaces) for Markov operators commuting with the modular automorphism group. This allows us to prove an individual ergodic theorem also in this setting. 
    In section 6, we provide explicit computations of F{\o}lner sequences which make the ergodic average dominance above true for groups of polynomial growth, and for the lamplighter group which is of exponential growth.

\section{Preliminaries}

\subsection{Noncommutative $L_p$-spaces in the Tracial Setting}
%We address the tracial case first, and then discuss the extension to the nontracial case at the end. 
Let $(\M, \tau)$ denote a semifinite von Neumann algebra with a semifinite faithful normal (s.f.n.) trace. 
Denote by $S_+$ the set of all $x \in \M_+$ such that $\tau(\supp(x)) < \infty$, where $\supp(x)$ is the  support projection of $x$. %(which must exist in $\M$ by virtue of $\M$ being a von Neumann algebra). 
Let $S$ denote the linear span of $S_+$. $S$ is a strongly dense involutive ideal of $\M$. For $p\in [1, \infty)$, $x \in S$, we define its $L_p$-norm 
$$\norm{x}_p := \left( \tau (\left| x \right| ^p) \right) ^{\frac{1}{p}}.$$ We define $L_p(\M)$ to be the $\norm{ \cdot}_p$-norm completion of $S$, and we identify $L_\infty(\M)$ with $ \M$. It is direct to see that $L_2(\M)$ is a Hilbert space with an inner product given by $\langle x, y \rangle := \tau \left( x y^* \right)$. For more background on noncommutative $L_p$-spaces, we refer to \cite{PX03}.

There is no notion of ``pointwise'' almost everywhere convergence per se in the noncommutative setting, but one has a counterpart of the notion of almost uniform convergence. This was explored in \cite{La76} (see \cite{Ja85} for more details). We recall the following definitions as stated in \cite{JX07}:
\begin{defn}\label{defintion au bau}
     A sequence of operators $\left(x_n\right)_{n \in \bN}\in L_p(\M)$ is said to converge \textit{bilaterally almost uniformly} (abbreviated as b.a.u.) to $x$ if for any $\varepsilon > 0$ there exists a projection $e \in \M$ such that $\tau(1_\M - e) < \varepsilon$ and $\norm{e ( x_n -x ) e}_\infty \to 0$. It is said to converge \textit{almost uniformly} (abbreviated as a.u.) to $x$ if for any $\varepsilon > 0$ there exists a projection $e \in \M$ such that $\tau(1_\M - e) < \varepsilon$ and $\norm{(x_n - x) e}_\infty \to 0$.
\end{defn}
It can be observed directly that a.u. convergence implies b.a.u. convergence, but not the converse. 
A counterexample is noted in \cite[Theorem 6.1 ]{HR24}.

\subsection{Convolutions of Measures, Random Walks, and the Markov Operator}
{We begin by recalling some basic facts for harmonic analysis on groups, for which we cite \cite{Fo16} as a source for further reading. Let $G$ be a locally compact second countable group with a left Haar measure $\lambda$. %that is, for every measurable set $E \subset G$, and every $g \in G$, $\lambda(gE) = \lambda(E)$. 
Let $\Delta$ denote the modular function of $G$, which is a continuous group homomorphism from $G$ to $\left( \mathbb{R}^+, \times \right)$, such that $\lambda(Eg) = \Delta(g) \lambda(E)$ for every measurable set $E \subset G$, and every $g \in G$. If we consider the right Haar measure given by $\tilde{\lambda}(E) = \lambda(E^{-1})$, then one may obtain that the Radon-Nikodym derivative $\frac{d\tilde{\lambda}}{d\lambda} = \Delta^{-1}$. We recall that if $f_1$ and $f_2$ are two functions in $L_1(G, \lambda)$, then their convolution product (which turns $L_1(G, \lambda)$ into a Banach $*$-algebra), is given by 
$$\left( f_1 * f_2 \right) (g) 
= \!\int \!\! f_1(h) f_2(h^{-1} g) d\lambda(h) 
= \!\int \!\! f_1(gh^{-1}) f_2(h) \Delta(h^{-1}) d\lambda(h) 
= \!\int \!\! f_1(gh^{-1}) f_2(h) d\tilde{\lambda}(h).$$
We call a group \textit{unimodular} if $\Delta \equiv 1$ is a constant function. %All compact groups, abelian groups, as well as discrete groups are unimodular. A very large class of Lie groups are also known to be unimodular. If $G/[G,G]$ is discrete, then $G$ is unimodular \cite{Fo16}. 
 % such that %the closure of the group generated by $\supp(\omega)$ is $G$. (We actually need something a bit less strict: we only need that the complement of the subgroup generated by the support of $\omega$ has measure $0$ with respect to the measure $\lambda$. This is strictly stricter than the subgroup being dense.) 
%Let us recall the definition of \textit{convolution} of measures \cite{Fo16}: 
If $\mu$ and $\nu$ are two finite Borel measures on $G$, then their convolution is defined as 
$$\left( \mu * \nu \right) (E) := \int_G \int_G 1_E(gh) d\mu(g) d\nu(h),$$ 
for any measurable subsets $E \subset G$. One may verify directly that if the measures $\mu$ and $\nu$ are absolutely continuous with respect to $\lambda$, then the measure $\mu * \nu$ is the unique measure whose density is given by $\frac{d\mu}{d\lambda} * \frac{d\nu}{d\lambda}$. 
Let $\omega$ be a probability measure on $G$,  denote by $\omega^{(n)}$ the convolution of $\omega$ with itself $n$ times. %One can verify through a direct computation via induction that 
%$\omega^{(n)}(E) = \int \int \dots \int 1_E ( g_1 g_2 \dots g_n ) d\omega(g_1) d\omega(g_2) \dots d\omega(g_n)$, and that 
%$\frac{d\omega^{(n)}}{d\lambda} = \left(\frac{d\omega}{d\lambda}\right)^{(n)}$.}

Let us consider a {$w^\star$-continuous} trace-preserving action $\alpha: G \curvearrowright \left( \M, \tau \right)$. We define the Markov operator on $\M$: for all $x \in\M$,
\begin{align}\label{T}
    T(x) := \int_G \alpha_g(x) d\omega(g).
\end{align}
%where the integral is the Bochner integral \cite{Bo33}. 
It is easy to see that $$T^n(x) = \int_G \alpha_g(x) d\omega^{(n)}(g),  \text{ for all } n \in \bN .$$ Moreover, one can check $T$ satisfies the following three conditions:
 \begin{enumerate}
    \item[(T1)] \label{Tprops1} $T$ is a contraction on $\M$: $\norm{T(x)} \leq \norm{x}$ for all $x \in \M$;
    \item[(T2)] \label{Tprops2} $T$ is positive: $T(x) \geq 0$ for all $x \geq 0$; and
    \item[(T3)] \label{Tprops3} $\tau \circ T \leq \tau$: $\tau\left(T(x)\right) \leq \tau(x)$ for all $x \in L_1(\M) \cap \M_+$.
\end{enumerate}
In \cite[Lemma 1.1]{JX07} it was shown that any $T$ that satisfies the above conditions extends to a positive normal contraction on $L_p(\M)$ for {$p \in [1, \infty]$}.

\subsection{The Junge-Xu Ergodic Theorems}

%\begin{thm}[Theorem 0.1 part (i), \cite{JX07}]\label{JungeXuMaximalIntegersBasic}
 %    Let $T$ be a map on $L_p(\M)$, $p \in \left(1, \infty \right)$ that satisfies the conditions \ref{Tprops1}, \ref{Tprops2}, \ref{Tprops3}. For all $x \in \left( L_p(\M) \right)_+$ there exists $a \in \left(L_p(\M)\right)_+$ such that for all $n \in \mathbb{N}$,
  %  $$M_n(x):=\frac{1}{n}\sum_{j=0}^{n-1}T^j(x) \leq a \textrm{  and  } \norm{a}_p \leq C_p \norm{x}_p,$$ 
  %  where $C_p$ is a positive constant depending only $p$. Moreover, $C_p \leq \frac{C_0 p^2}{(p-1)^2}$, with $(p-1)^{-2}$ the optimal order as $p \to 1$. 
%\end{thm}

In noncommutative $L_p$-spaces, there are certain general difficulties in formulating maximal inequalities in a form similar to that for the commutative counterpart, as outlined in \cite{JX07}. 
The following notions were introduced by Pisier \cite{Pi98} for hyperfinite von Neumann algebras, and generalised by Junge \cite{Ju02}.
For all $p \in [1, \infty]$, we define $L_p(\M; \ell_\infty)$ to be the space of all sequences $x = \left(x_n\right)_{n\in \mathbb{N}}$ in $L_p(\M)$ which admit a factorisation of the following form: 
there exist $a,b \in L_{2p}(\M)$ and a bounded sequence $\left(y_n\right)_{n \in \mathbb{N}}$ in $\M$ such that for each $n\in \mathbb{N}$, $x_n = a y_n b$. 
We define the norm 
$$\norm{x}_{L_p(\M; \ell_\infty)} = \inf \setof{\norm{a}_{2p} \sup_{n \in \mathbb{N}}\norm{y_n}_\infty \norm{b}_{2p}},$$
where the infimum is taken over all such factorisations of $x_n$ in terms of $a$, $b$ and $y_n$. 
One can also prove in a straightforward manner that if $x=\left(x_n\right)_{n\in \mathbb{N}}$ is a positive sequence, that is, if $x_n \geq 0$ for all $n \in \mathbb{N}$, then $x \in L_p(\M; \ell_\infty)$ iff there exists $a \in L_p^+(\M)$ such that $x_n \leq a$ for all $n\in\mathbb{N}$. 
Furthermore, 
\begin{equation}\label{eq: charac for sup norm}
 \norm{x}_{L_p(\M; \ell_\infty)} = \inf \setof{\norm{a}_{p} \textrm{ } \vert \textrm{ } a \in L_p^+(\M) \textrm{ and } x_n\leq a \;, \forall n \in \mathbb{N}}.   
\end{equation} 
The standard convention is to denote $\norm{x}_{L_p(\M ; \ell_\infty)}$ as $\norm{\sup_n^+x_n}_p$. 
%In this notation, one may formulate the following maximal inequality, which reduces to \cref{JungeXuMaximalIntegersBasic} when one restricts to the positive cone of $L_p(\M)$: 
Now we state the following results from \cite{JX07} and \cite{Ye77} on maximal and individual ergodic theorems for integer actions on noncommutative $L_p$-spaces, which we will use to deduce the maximal and individual counterparts for amenable groups. 
For a map $T$ on $L_p(\M)$, we define the ergodic average for $T$ denoted by $M_n:L_p(\M) \to L_p(\M)$, $n\in\mathbb{N}$ to be $M_n(x) := \frac{1}{n}\sum_{j=0}^{n-1}T^j(x)$. 
Then, the maximal ergodic theorem can be stated as: 
\begin{prop}[Theorem 4.1, \cite{JX07}]\label{JungeXuMaximalIntegers}
     Let $T$ be a map on $L_p(\M)$, $p \in \left(1, \infty \right)$ that satisfies the conditions %\ref{Tprops1}, \ref{Tprops2}, \ref{Tprops3}. 
     (T1), (T2), and (T3). 
     For all $x \in L_p(\M)$,
    $$\norm{ \sup_n{}^+ M_n(x)}_p \leq C_p \norm{x}_p,$$ 
    with $C_p \leq \frac{C_0 p^2}{(p-1)^2}$, this being the optimal order of $C_p$ as $p \to 1$. 
\end{prop}

\begin{prop}[Theorem 1, \cite{Ye77}]\label{YeadonMaximalIntegers}
    Let $T$ be a map on $L_1(\M)$ that satisfies the conditions %\ref{Tprops1}, \ref{Tprops2}, and \ref{Tprops3}. 
    (T1), (T2), and (T3). 
    Then, for all $x \in L_1(\M)$ and $\eps>0$, there exists a projection $e \in \M$ such that 
    \begin{align*}
        \norm{eM_n(x)e}_\infty < 4 \eps \text{ for any } n \in \mathbb{N}, \;
        &\textrm{ and } \tau(1_\M-e) < \frac{4}{\eps} \norm{x}_1.
    \end{align*}
\end{prop}

We now look at the following pair of definitions:
\begin{defn}\label{StrongAndWeakTypes}
    Let $p \in [1,\infty]$, and $\left(S_i\right)_{i \in I}$ be a family of maps from $L_p^+(\M)$ to $L_0 (\M)$ where the latter denotes the space of all positive measurable operators. 
    Then, 
    \begin{enumerate}
        \item[(i)] for $p\in[1,\infty)$, we say that the family $\left(S_i\right)_{i \in I}$ is of weak type $(p,p)$ with constant $C$ if there exists a constant $C>0$ such that for all $x \in L_p^+(\M)$ and $\lambda>0$, there exists a projection $e \in \M$ satisfying
        \begin{align*}
            \tau(1_\M - e) \leq \frac{C^p}{\lambda^p}\norm{x}_p^p\; &\textrm{ and } \;
            eS_i(x)e \leq \lambda e, \;\; \forall i \in I;
        \end{align*}

        \item[(ii)] for $p\in[1,\infty]$, we say that the family $\left(S_i\right)_{i \in I}$ is of strong type $(p,p)$ with constant $C$ if there exists a constant $C>0$
        \begin{align*}
            \norm{\sup_{i\in I}{}^+S_i(x)}_p \leq C \norm{x}_p, 
           \;\; \forall x \in L_p(\M).
        \end{align*}
    \end{enumerate}
\end{defn}

Thus, Proposition \ref{JungeXuMaximalIntegers} can be restated by saying that the family of operators $(M_n)_{n\in \mathbb{N}}$ is of strong type $(p,p)$ with constant $C_p$, and Proposition \ref{YeadonMaximalIntegers} is equivalent to saying that $(M_n)_{n\in \mathbb{N}}$ is of weak type $(1,1)$ with constant $4$.

For any $p \in (1, \infty)$, the map $T$ imposes a canonical decomposition $L_p(\M)= \F_p^T \oplus (\F_p^T)^\perp$, where $\F_p^T = \setof{ x \in L_p(\M) \mid T(x)=x }$, $(\F_p^T)^\perp = \overline{\left(1-T\right)L_p(\M)}^{\norm{\cdot}_p}$. 
For conjugate indices $p,q$ (i.e. $\frac{1}{p} + \frac{1}{q} = 1$), we have $(L_p(\M))^*=L_q(\M)$, and consequently $(\F_p^T)^*=\F_q^T$. Denote the positive contractive projection from $L_p(\M)$ to $\F_p^T$ by $\mathcal{P}$.
%Then for any indices $p$ and $q$, $F_p$ and $F_q$ coincide on $L_p(\M) \cap L_q(\M)$. 
From the theorems above, the following individual ergodic theorem may be obtained as a corollary:
\begin{cor}[Theorem 1, \cite{Ye77} for $p=1$; Corollary 6.4 and Remark 6.5, \cite{JX07} for $p \in (1,\infty)$ ]\label{JungeXuIntegers}
     Let $T$ be a map on $L_p(\M)$, $p \in \left(1, \infty \right)$ that satisfies the conditions 
     (T1), (T2), and (T3), 
     %\ref{Tprops1}, \ref{Tprops2}, \ref{Tprops3}, 
     then $(M_n(x))_n$ converges to $\mathcal{P}(x)$ b.a.u. for $p \in [1, \infty)$. Furthermore, for $p \in [2, \infty)$, the convergence is a.u.
\end{cor}

\subsection{Unimodular Amenable Groups} 

The F{\o}lner condition was originally formulated for discrete groups in \cite{Fo55}. An equivalent condition was shown to hold for locally compact second countable groups by Emerson and Greenleaf in \cite{EG67}: 
%\begin{prop}[Theorem 1.4.1, \cite{EG67}]\label{Folner sets locally compact}
  %  Let $G$ be a locally compact group. Then, the following are equivalent:
   %  \begin{enumerate}[label=(\roman*)]
    %    \item $G$ is amenable.
    %    \item For every $\varepsilon >0$ and every compact $K \subset G$, there exists a compact set $F$ with $0 < \lambda(F) < \infty$ such that $\frac{\lambda(g F \triangle F)}{\lambda(F)} < \varepsilon$ for every $g \in K$.
    %    \item For every $\varepsilon >0$ and every compact set $K\subset G$  such that $e \in K$ and $0<\lambda(K) <\infty$, there exists a compact set $F$ with $0 < \lambda(F) < \infty$ such that $\frac{\lambda(K F \triangle F)}{\lambda(F)} < \varepsilon$.
   % \end{enumerate}
%\end{prop}

\begin{prop}[Theorem 3.2.1, \cite{EG67}]\label{Folner sequences locally compact}
    Let $G$ be a locally compact and second countable group. Then $G$ is amenable iff there exists a sequence of compact subsets $\{ F_n \}_{n \in \mathbb{N}}$ of $G$ satisfying $F_n \subset F_{n+1}$ and $\bigcup_{n \in\mathbb{N}}F_n = G$ such that for any compact $K \subset G$, $$\lim_{n\rightarrow \infty} \frac{\lambda(KF_n \triangle F_n)}{\lambda(F_n)} = 0.$$%, (which is equivalent to $\lim \frac{\lambda(KF_n \cap F_n)}{\lambda(F_n)} = 1$, also equivalent to $\lim \frac{\lambda(KF_n)}{\lambda(F_n)} = 1$).
\end{prop}
The limit above is additionally equivalent to $\lim_{n\rightarrow \infty} \frac{\lambda(KF_n \backslash F_n)}{\lambda(F_n)} = 0$, or $\lim_{n\rightarrow \infty} \frac{\lambda(KF_n \cap F_n)}{\lambda(F_n)} = 1$, also equivalent to $\lim_{n\rightarrow \infty} \frac{\lambda(KF_n)}{\lambda(F_n)} = 1$. We will freely use these characterisations throughout the paper and refer the reader to \cite{Pa88} for further details. 
Every locally compact second countable amenable group has a left-F{\o}lner sequence with respect to the left Haar measure $\lambda$. It is direct to see for every left-F{\o}lner sequence $\{F_n\}_{n \in \mathbb{N}}$, the sequence $\{F_n^{-1}\}_{n \in \mathbb{N}}$  is right-F{\o}lner sequence with respect to the right Haar measure $\tilde{\lambda}$.  
%$\lim \frac{\tilde{\lambda}(F'_n K \backslash F'_n)}{\tilde{\lambda}(F'_n)} = \lim \frac{\lambda(K^{-1}F_n'^{-1} \backslash F_n'^{-1})}{\lambda(F_n'^{-1})} = \frac{\lambda(K^{-1} F_n \backslash F_n)}{\lambda(F_n)} = 0 $. 
%One may ask if a sequence may be left and right-invariant at the same time. Now, if $F_n = F_n^{-1}$, then this sequence is both left and right-invariant.
%one has a two-sided  F{\o}lner sequence. 
One may see by a comment noted in \cite{CW22} that if $G$ is nonunimodular, then there does not exist a right-F{\o}lner sequence with respect to the left Haar measure $\lambda$: If $\Delta(g)>1$, then for any $F\subset G$, $\lambda(Fg\backslash F) \geq \lambda(Fg) - \lambda( F) = \left( \Delta(g) - 1 \right) \lambda(F)$, and $\frac{\lambda(Fg\backslash F)}{\lambda(F)} \geq \Delta(g) - 1$, so the F{\o}lner invariance criterion is not satisfied from the right with respect to the left Haar measure. However, if $G$ is unimodular, then one may work around this. %Any symmetric F{\o}lner sequence can be seen to be left and right invariant with respect to the Haar measure in that case. One may go a step further and ask for bi-invariance:
\begin{prop}[Proposition 2, \cite{OW87}]\label{biinvariant Folner sets existence}
    Let $G$ be a locally compact second countable unimodular amenable group. Then, there exists a sequence of compact subsets $\{ F_n\}_{n \in \mathbb{N}}$ of $G$ such that for any compact $K \subset G$, $$\lim _{n\rightarrow \infty}  \frac{\lambda(KF_nK \backslash F_n)}{\lambda(F_n)} = 0$$ 
    which is equivalent to $\lim_{n\rightarrow \infty}  \frac{\lambda(K F_n K \cap F_n)}{\lambda(F_n)} = 1$, also equivalent to $\lim _{n\rightarrow \infty} \frac{\lambda(KF_nK)}{\lambda(F_n)} = 1$.
\end{prop}

\section{The ergodic average dominance}

In this section, we develop the proof of our ergodic average dominance for unimodular amenable group acting on semifinite von Neumann algebras, formulated in Theorem \ref{action dominance}.
Unless explicitly stated otherwise, throughout this section, $G$ denotes a locally compact second countable unimodular amenable group with Haar measure $\lambda$, $\M$ is a semifinite von Neumann algebra with a s.f.n. trace $\tau$, and $\alpha: G \curvearrowright \left(\M, \tau \right)$ is a {$w^\star$-continuous} trace-preserving action. By this we mean that for every $g \in G$, $\alpha_g$ is an automorphism on $\M$ such that  $\alpha_g \circ \tau = \tau$,
and that for every $x \in \M$, the map $g \mapsto \alpha_g(x)$ is continuous with respect to the $w^\star$-topology of $\M$. It then follows that for every $x \in L_p(\M)$ with $p \in [1, \infty)$, the map $g \mapsto \alpha_g(x)$ is weakly measurable, and as shown following Lemma 1.1 of \cite{JX07}, such an action extends to an action on $L_p(\M)$ by isometries. 

\subsection{Dominating the Ergodic Averages}
The following theorem is the main result of our paper: 
\begin{thm}\label{action dominance}
    For every two-sided F{\o}lner sequence of $G$, there exists a subsequence %of compact subsets of $G$ 
    $\{F_n\}_{n \in \mathbb{N}}$, a constant $C > 0$, a strictly increasing function $N:\mathbb{N} \to \mathbb{N}$, and 
    %a  $T$ satisfying conditions \ref{Tprops1}, \ref{Tprops2}, \ref{Tprops3} 
    a positive linear map $T$ satisfying conditions 
    (T1), (T2), and (T3) 
    %\ref{Tprops1}, \ref{Tprops2}, \ref{Tprops3}
    such that for any $x$ in $L_p^+(\M)$, {$p \in [1, \infty]$} and %any 
    large enough $n$,
    \begin{align}\label{Ergodic Operator Inequality}
    \frac{1}{\lambda{(F_n)}} \int_{F_n} \alpha_g(x) d\lambda(g) &\leq C \frac{1}{N(n)}\sum_{j=0}^{N(n)-1} T^j(x).\tag{AD}
    \end{align} 
    Moreover, $T$ can be realised as a Markov operator $T(x)= \int_G \alpha_g(x) d\omega(g)$ for some probability measure $\omega$ on $G$
    absolutely continuous  
    with respect to $\lambda$, i.e.
    $\omega \in L_1^+(G, \lambda)$, and $\omega$ depends only on $\{ F_n \}_{n \in \mathbb{N}}$.
\end{thm}

The proof of the theorem above will be given in the next subsection. Before that, we first discuss a few preparatory results.

 The existence of two-sided F{\o}lner sequences of $G$ is guaranteed by Proposition \ref{biinvariant Folner sets existence}. 
    Without loss of generality, we may also assume our F{\o}lner sets to be symmetric 
    and containing the group identity (i.e. $F_n = F_n^{-1}$ and $e \in F_n$ for all $n\in \bN$).
    This is because, we may define $F'_n = F_n \cup F_n^{-1} \cup \{ e \}$. 
    Since $G$ is unimodular, if $F_n$ is a two-sided  F{\o}lner sequence, then so is $F'_n$, and one has that
    $\lambda(F'_n) \leq 2 \lambda(F_n) + \lambda(\{e\}) \leq 3 \lambda(F_n)$, 
    and hence for any $x\in L_p^+(\M)$,
    \begin{align*}
    \frac{1}{\lambda{(F_n)}} \int_{F_n} \alpha_g(x) d\lambda(g) 
    \leq \frac{1}{\lambda{(F_n)}} \int_{F'_n} \alpha_g(x) d\lambda(g) 
    \leq \frac{3}{\lambda{(F'_n)}} \int_{F'_n} \alpha_g(x) d\lambda(g).
    \end{align*}
   We note that with $T(x)= \int \alpha_g(x) d\omega(g)$, we may rewrite \eqref{Ergodic Operator Inequality} as: 
$$
   \frac{1}{\lambda{(F_n)}} \int_{F_n} \alpha_g(x) d\lambda(g) 
   %&
   \leq \frac{C}{N(n)}\sum_{j=0}^{N(n)-1} \int_G \alpha_g(x) d\omega^{(j)}(g), $$
   which is equivalent to 
       \begin{equation}\label{eq: main inequality} 
       \int_G \alpha_g(x) \left( \frac{C}{N(n)}\sum_{j=0}^{N(n)-1} \frac{d\omega^{(j)}}{d\lambda}(g) - \frac{\chi_{F_n}(g)}{\lambda{(F_n)}} \right) d\lambda(g)\geq 0.
        \end{equation}
%where $\frac{d\omega}{d\lambda}$ corresponds to the Radon-Nikodym derivative \cite{Ni30, Co13} of the $j^{th}$ convolution of $\omega$ with respect to $\lambda$. 
In particular, if the following inequality is true, we will get \eqref{eq: main inequality}.
%$$\frac{C}{N(n)}\sum_{j=0}^{N(n)-1} \frac{d\omega^{(j)}}{d\lambda}(g) \geq \frac{\chi_{F_n}(g)}{\lambda{(F_n)}}, \textrm{ for } \lambda\textrm{-a.e. }g \in F_n.$$ 
   \begin{equation}\label{Measure Comparison Inequality}
     C\frac{\lambda{(F_n)}}{N(n)} \sum_{j=0}^{N(n)-1} \frac{d\omega^{(j)}}{d\lambda}(g) \geq 1, \textrm{ } 
     \textrm{ for } \lambda\textrm{-a.e. }g \in F_n,
 \end{equation}
In fact, it can be shown that if \eqref{eq: main inequality} is true for any $G$, any action $\alpha:G \curvearrowright M$ on a von Neumann algebra with a s.f.n.trace and any element $x\in L_p^+(\M)$, then it also implies \eqref{Measure Comparison Inequality}. % This can be seen easily by taking $f=\frac{C}{N(n)}\sum_{j=0}^{N(n)-1} \frac{d\omega^{(j)}}{d\lambda} - \frac{\chi_{F_n}}{\lambda{(F_n)}}$ in the following lemma. 

Thus, the problem of finding a Markov operator in Theorem \ref{action dominance} is transferred to the problem of finding a probability measure $\omega$ on $G$ which satisfies \eqref{Measure Comparison Inequality}. This is our alternative for the classical transference principle.

\subsection{Comparison of Measures}
This subsection is dedicated to the construction of a probability measure $\omega$, a function $N$, a family of compact sets $F_n$, and a constant $C$ such that 
 \eqref{eq: main inequality} holds. 

 For groups of polynomial growth, an immediate choice is to take  $F_n$'s to be the balls of radius $n$, and $\omega$ to be the uniform probability measure on the unit ball, as was demonstrated in \cite{HLW21}. Then the comparison \eqref{eq: main inequality} for groups of polynomial growth follows from 
 a Gaussian lower bound for the density of convolutions for such groups as obtained in \cite{HSC93}.  
Unfortunately, no such Gaussian lower bounds %\cref{Gaussian lower bound polynomial growth} 
exist for groups of a higher growth rate than polynomial growth, in particular for general amenable groups whose growth rate might be exponential. This is where we need to develop a new construction, which is at the heart of this paper. Instead of uniform probability measures on the unit ball, we shall be using a convex combination of uniform probability measures on a sequence of sets whose dynamical interiors are F{\o}lner sets for our $\omega$. This construction is motivated by \cite{KV83}. 

First we state a couple of results that are building blocks in our proof. The following is a result in elementary calculus, that we include for the sake of completeness. Unless explicitly stated otherwise, all limits henceforth are for $n \to \infty$.

\begin{lem}\label{lemma on limits}
    If $\left(a_n\right)_{n \in \mathbb{N}}$ is a sequence in $[0,1]$ such that $\lim_{n\rightarrow\infty}a_n=0$
     and $\left(b_n\right)_{n \in \mathbb{N}}$ is a sequence such that $\lim_{n\rightarrow\infty}b_n=+\infty$, then 
    $$\lim_{n \to \infty} (1-a_n)^{b_n} = e^{-\lim _{n \to \infty} a_n b_n},$$ 
    if $\lim_{n \to \infty}  {a_n b_n}$ exists.
\end{lem}
\begin{proof}
    Let us define $c_n = (1-a_n)^{b_n}$. Then,
\begin{align*}
    \log(c_n) 
    = b_n \log(1-a_n)
    = b_n \left( - \sum_{j=1}^\infty \frac{\left(a_n\right)^j}{j} \right)
    = - a_n b_n \left( 1 + \frac{1}{2} a_n + \frac{1}{3} a_n^2  + \dots \right).
\end{align*}
There are exactly three possibilities at this point:
 \begin{enumerate}[label=(\roman*)]
    \item If $\lim a_n b_n = 0$, then $\lim \log(c_n) = 0$, that is $\lim c_n = 1$.

    \item If $\lim a_n b_n = +\infty$, then $\lim \log(c_n) = - \infty$, that is $\lim c_n = 0$.

    \item If $\lim a_n b_n = c \in \left(0, +\infty \right)$, then $\lim \log(c_n) = - c$, that is $\lim c_n = e^{-c}$.
\end{enumerate}
This proves the claim.
\end{proof}

The followings are a couple of definitions that might be known to some in the field of analysis on groups, that we state for the sake of completeness.

\begin{defn}\label{interior}
    Let $H,K \subset G$ be compact. For all $h \in H$, we define the 
    \textit{left dynamical interior of $K$ w.r.t. $H$} to be $\iota_l(H,K) := \{g \in K \mid Hg \subset K \}$. 
    Similarly we define the \textit{right dynamical interior}: 
    $\iota_r(H,K) := \{g \in K \mid gH \subset K \}$. 
    And finally, we may define the \textit{bilateral dynamical interior} for $K$ with respect to $H_1, H_2$, all compact sets: 
    $\iota(H_1, H_2, K):= \{ g \in K \mid H_1 g H_2 \subset K \}$.
\end{defn}
Note that our definition of dynamical interior is slightly different compared to some of the other existing notions, e.g. \cite{KL16}. 

\begin{rem}\label{interiors characterisation}
    It is easy to see that 
    $\iota_l(H,K) 
    = K \cap \bigcap_{h \in H} h^{-1}K 
    = \bigcap_{h \in H} {\iota_l(h,K)}$, 
    where we denote $\setof{h}$ by $h$ to simplify the notation. 
    Similarly, we have 
    $\iota_r(H,K) 
    = K \cap \bigcap_{h \in H} K h^{-1} 
    = \bigcap_{h \in H} {\iota_r(h,K)}$, and 
    $\iota(H_1,H_2,K) 
    = K \cap \bigcap_{h_1 \in H_1, h_2 \in H_2} h_1^{-1}K h_2^{-1} 
    = \bigcap_{h_1 \in H_1, h_2 \in H_2} {\iota(h_1,h_2,K)}$.
\end{rem}

With this notion of dynamical interior, we can state the following %observation:
lemma: 
\begin{lem}\label{convolution absorption for two}
    Let $G$ be a locally compact second countable group and  $H,K \subset G$ be compact subsets. %and almost every 
    For $g \in K$, $\left(  \frac{\chi_H}{\lambda( H )} * \chi_K \right) (g) = 1$ iff $g \in \iota_l(H_1^{-1},K)$ where $H_1 \subset H$ with $\lambda(H\backslash H_1)=0$.
    {Furthermore, if $G$ is unimodular, then} $\left( \chi_K *  \frac{\chi_H}{\lambda(H)}  \right) (g) = 1$ iff $g \in \iota_r(H_2^{-1},K)$ where $H_2 \subset H$ with $\lambda(H\backslash H_2)=0$.
\end{lem}
\begin{proof}
    Using the definition of convolution, we have that 
    $$\left( \frac{\chi_H}{\lambda(H)}  * \chi_K \right) (g)
    = \frac{1}{\lambda(H)} \int_G \chi_H (h) \chi_K (h^{-1} g) d\lambda(h) = \frac{1}{\lambda(H)} \int_H \chi_K (h^{-1} g) d\lambda(h).$$ 
    This is equal to $1$ iff $h^{-1}g \in K$ for $\lambda$-a.e. $h \in H$, which is  true iff there exists $H_1 \subset H$ with $\lambda(H \backslash H_1)=0$ such that  $g \in \iota_l(H_1^{-1},K)$. 
    
    For the second part, we have 
    $$\left( \chi_K  *  \frac{\chi_H}{\lambda(H)}\right) \! (g)
        = \frac{1}{\lambda(H)} \!\int _G \chi_K (h) \chi_H (h^{-1} g) d\lambda(h) 
        = \frac{1}{\lambda(H)} \!\int_G  \chi_K (g k^{-1}) \chi_H (k) \Delta(k^{-1}) d\lambda(k).$$ 
        If $G$ is unimodular, we have $\Delta(k^{-1}) \!=\! 1$ for all $k$, and this can be written as $\frac{1}{\lambda(H)} \!\int\! \chi_K (g k^{-1}) \chi_H (k) d\lambda(k) 
        \!=\! \frac{1}{\lambda(H)} \int_H \chi_K (g k^{-1}) d\lambda(k)$. This is equal to $1$ iff $g k^{-1} \!\in\! K$  for $\lambda$-a.e. $k \!\in\! H$ which  is true iff there exists $H_2 \subset H$ with $\lambda(H \backslash H_2)=0$ such that  $g \in \iota_r(H_2^{-1},K)$. 
\end{proof}
We observe from the lemma above that the convolution operation on the characteristic function of the set $K$ with respect to the probability measure on the set $H$ still yields $1$ on the dynamical interior of $K$. 
By an induction argument, we get the corollary below which allows us to absorb multiple convolution products from both sides.

\begin{cor}\label{convolution absorption for n}
    {Let $G$ be a locally compact second countable unimodular group.} For any $ n\in \mathbb{N}$, any compact sets $K_1, \dots , K_n \subset G$, %$g \in G$, and 
    any $j \in \setof{1, \dots , n}$, and $g \in K_j$, the following are equivalent:
     \begin{enumerate}[label=(\roman*)]
        \item The convolution product yields $1$: $$\left(  \frac{\chi_{K_1}}{\lambda( K_1 )}  * \dots *  \frac{\chi_{K_{j-1}}}{\lambda( K_{j-1} )}  * \chi_{K_j} * \frac{\chi_{K_{j+1}}}{\lambda( K_{j+1} )}  * \dots *    
        \frac{\chi_{K_n}}{\lambda( K_n )} \right) (g) = 1.$$
        
        \item For every $i \in \setof {1, \dots , n }$, there exists $K_i'\subset K_i$ with $\lambda(K_i \backslash K_i') = 0$, such that $$g \in \bigcap_{k_1 \in K_1'} \dots \bigcap_{k_{j-1} \in K_{j-1}'} \bigcap_{k_{j+1} \in K_{j+1}'} \dots \bigcap_{k_n \in K_n'} k_1 \dots k_{j-1} K_j k_{j+1} \dots k_n.$$ 

        \item For every $i \in \setof{1, \dots , n }$, there exists $K_i'\subset K_i$ with $\lambda(K_i \backslash K_i') = 0$, such that 
        $$\left( K_1' \dots K_{j-1}' \right)^{-1} g \left( K_{j+1}' \dots K_n' \right)^{-1} \subset K_j.$$ 

        \item For every $i \in \setof{1, \dots , n }$, there exists $K_i'\subset K_i$ with $\lambda(K_i \backslash K_i') = 0$, such that %The element $g$ belongs to 
        $$g \in \iota\left({\left( K_1' \dots K_{j-1}' \right)^{-1}}, {\left( K_{j+1}' \dots K_n' \right)^{-1}}, K_j \right).$$
    \end{enumerate}
\end{cor}

The corollary above motivates our eventual construction of $\omega$ in that if the dynamical interior of a set is proportionally large enough (which can happen for F{\o}lner sets), then it remains approximately invariant under convolution. 
We are now finally in a position to embark on the proof of \eqref{Measure Comparison Inequality}.  

Let $G$ be a locally compact second countable unimodular group. Let $\{ F_n \}_{n \in \mathbb{N}}$ be a family of symmetric compact subsets of $G$, each containing the group identity. For a given increasing function $N: \mathbb{N} \rightarrow \mathbb{N}$ such that $N(2)>2$, 
define another family $\{E_n\}_{n \in \mathbb{N}}$ by $E_1 = F_1$ and 
\begin{equation}\label{eq: def of En}
   E_n = \left(E_{n-1}^{N(n)-2}\right)^{-1} F_n \left(E_{n-1}^{N(n)-2}\right)^{-1} \quad \text{ for }n \geq 2.
\end{equation}
Let $t_n \in (0,1)$ for $n \in \mathbb{N}$ such that $\sum_{n \in \mathbb{N}} t_n = 1$. We define $\omega$ to be the unique probability measure on $G$ such that 
\begin{equation}\label{eq: def of w}
\frac{d \omega}{d \lambda} = \sum_{n \in \mathbb{N}} t_n \frac{\chi_{E_n}}{\lambda(E_n)}.
\end{equation}

\begin{prop}\label{measure comparison in terms of sizes}
    %Then, for all $g \in F_n$, \begin{align*}
%    \lim \frac{\lambda(F_n)}{N(n)} \sum_{j=0}^{N(n)-1} \frac{d \omega^{(j)}}{d \lambda} (g)
 %   \geq \left( \lim \frac{\lambda(F_n)}{\lambda(E_{n})} \right) \left( 1 - \left( \lim \frac{r_{n+1}}{r_n} \right)\right) \left( \frac{1 - e^{-\lim r_n N(n)}}{\lim r_n N(n)} - e^{-\lim r_n N(n)} \right),
%\end{align*}
%where $r_n := \sum_{j=n}^\infty t_j$.
    We keep the same notation as above. Set $r_n = \sum_{j=n}^\infty t_j$, if we choose $t_n$ and $N(n)$ such that $\lim _{n\rightarrow \infty}r_n N(n)\in (0, +\infty)$ and  $\lim _{n\rightarrow \infty}\frac{r_{n+1}}{r_n} \in (0, 1)$, 
    then 
    there exists a constant $C'>0$ such that, 
    %for all $g \in F_n$, 
    $$\liminf _{n\rightarrow \infty}\frac{\lambda(F_n)}{N(n)} \sum_{j=0}^{N(n)-1} \frac{d \omega^{(j)}}{d \lambda} %(g)
    \geq C' \left(\liminf_{n\rightarrow \infty} \frac{\lambda(F_n)}{\lambda(E_{n})} \chi_{F_n} \right).$$ 
   % In particular, choosing $t_n = \frac{1}{2^n}$ and $N(n) = 2^n$, we have that for all $g \in F_n$, $$\liminf_{n\rightarrow \infty}\frac{\lambda(F_n)}{N(n)} \sum_{j=0}^{N(n)-1} \frac{d \omega^{(j)}}{d \lambda} (g)\geq \frac{1}{4}\left(1 - \frac{3}{e^2} \right) \left( \liminf_{n\rightarrow \infty} \frac{\lambda(F_n)}{\lambda(E_{n})} \right).$$
\end{prop}

\begin{proof}
    A direct computation gives that 
    \begin{align}\label{eq: conv sum}
    \frac{d \omega^{(j)}}{d \lambda} = \sum_{i_1, \dots, i_j \in \mathbb{N}} t_{i_1} \dots t_{i_j} \frac{\chi_{E_{i_1}}}{\lambda(E_{i_1})}  * \dots * \frac{\chi_{E_{i_j}}}{\lambda(E_{i_j})}.
    \end{align}
    Now we shall obtain lower bounds for $\frac{d \omega^{(j)}}{d \lambda}$ using Corollary \ref{convolution absorption for n}. Suppose for the $j$ indices $i_1, \dots, i_j \in \mathbb{N}$, we consider only those where the highest index is $n$, and it occurs only once. Then, we have
    \begin{align}\label{eq: lower estimate}
    \begin{split}
\frac{d \omega^{(j)}}{d \lambda}
    &\geq \sum_{i_1, \dots, i_{j-1} < n}t_{i_1} \dots t_{i_{j-1}} t_{n} \frac{1}{\lambda(E_{n})} 
    \left( \chi_{E_{n}} * \frac{\chi_{E_{i_1}}}{\lambda(E_{i_1})}  * \dots * \frac{\chi_{E_{i_{j-1}}}}{\lambda(E_{i_{j-1}})}  \right.\\
    & \quad \; \;+ \frac{\chi_{E_{i_1}}}{\lambda(E_{i_1})} * \chi_{E_{n}} *  \frac{\chi_{E_{i_2}}}{\lambda(E_{i_2})}  * \dots * \frac{\chi_{E_{i_{j-1}}}}{\lambda(E_{i_{j-1}})} + \dots\\
    %&+ \frac{\chi_{E_{i_1}}}{\lambda(E_{i_1})}  *  \dots *  \frac{\chi_{E_{i_{j-2}}}}{\lambda(E_{i_{j-2}})} * \chi_{E_{n}} * \frac{\chi_{E_{i_{j-1}}}}{\lambda(E_{i_{j-1}})} \\
    & \quad \;\;\left. + \frac{\chi_{E_{i_1}}}{\lambda(E_{i_1})}  *  \dots *  \frac{\chi_{E_{i_{j-1}}}}{\lambda(E_{i_{j-1}})} * \chi_{E_{n}}
    \right).    
\end{split}   
\end{align}
We may observe from \eqref{eq: def of En} that $e \in E_n = E_n^{-1} \subset E_{n+1}$ for all $n \in \mathbb{N}$, and that 
$$F_n = \iota\left({\left(E_{n-1}^{N(n)-2}\right)^{-1}}, {\left(E_{n-1}^{N(n)-2}\right)^{-1}}, {E_n}\right).$$ 
Hence, for every $g \in F_n$, by Corollary  \ref{convolution absorption for n}, every convolution product in the sum above equal $1$, and we are left with 
\begin{align*}
    \frac{d \omega^{(j)}}{d \lambda}(g)
    \geq \sum_{i_1, \dots, i_{j-1} < n}t_{i_1} \dots t_{i_{j-1}} t_{n} \frac{1}{\lambda(E_{n})} j
    = \left( t_1 + \dots + t_{n-1} \right)^{j-1} t_{n} j \frac{1}{\lambda(E_{n})}.
\end{align*} 
 %Defined thus, $\left(r_n\right)_{n \in \mathbb{N}}$ is a monotonically decreasing sequence with $r_1=1$ and $\lim%_{n \to \infty} 
%r_n = 0$. $r_n$ and $t_n$ are related thus: $t_n = r_n - r_{n+1}$, and $t_1 + \dots + t_{n} = 1 - r_{n+1}$.
We rewrite the term on the right-hand side of the inequality above using $r_n$, then we have
\begin{align*}
    \frac{d \omega^{(j)}}{d \lambda}(g)
    &\geq \left( 1 - r_{n} \right)^{j-1} \left( r_{n} - r_{n+1} \right) j \frac{1}{\lambda(E_{n})},
\end{align*}
%Now, choosing the greatest term in our truncated sum $n'$ as $n$ and defining $F_n$ to be the set of all $g \in E_n$ such that $\left(E_{n-1}^{N(n)-2}\right)^{-1} g \left(E_{n-1}^{N(n)-2}\right)^{-1} \subset E_n $,
and
\begin{align*}
    \sum_{j=0}^{N(n)-1} \frac{d \omega^{(j)}}{d \lambda} (g) \geq \frac{1}{\lambda(E_{n})} \left( r_{n} - r_{n+1} \right) \sum_{j=0}^{N(n)-1} \left( 1 - r_{n} \right)^{j-1} j.
\end{align*}
The sum $\sum_{j=0}^{N(n)-1} \left( 1 - r_{n} \right)^{j-1} j$ can now be computed exactly using the summation formula for an arithmetico-geometric sequence:%, which we recall here for sake of completeness: if we have an arithmetic sequence $A_n = a + (n-1)d$ and a geometric sequence $G_n = br^{n-1}$, then $\sum_{j=1}^n A_j G_j = \frac{ab-(a+nd)br^n}{1-r} + \frac{dbr(1-r^n)}{(1-r)^2}$. In our case, for $\sum_{j=0}^{N(n)-1} j \left( 1 - r_{n} \right)^{j-1}$, we have $a = 1$, $b = 1$, $d = 1$, $r = (1-r_n)$. So, 
\begin{align*}
    \sum_{j=0}^{N(n)-1} j \left( 1 - r_{n} \right)^{j-1} = \frac{1 - r_n N(n) (1-r_n)^{N(n)-1} - (1-r_n)^{N(n)}}{r_n^2} .
\end{align*}
Hence,
\begin{align*}
    \frac{\lambda(F_n)}{N(n)} \sum_{j=0}^{N(n)-1} \frac{d \omega^{(j)}}{d \lambda} (g)
    \geq \frac{\lambda(F_n)}{\lambda(E_{n})} \left( 1 - \frac{r_{n+1}}{r_n} \right) \left( \frac{1 - (1-r_n)^{N(n)}}{r_n N(n)}  - (1-r_n)^{N(n)-1} \right).
\end{align*}

We now need to evaluate this limit as $n \to \infty$.  %The following lemma is elementary, but we still provide a proof for the sake of completeness, and the main idea behind the proof is somewhat useful. In the following, all limits for $n$ are taken to $\infty$.
We assume that $\lim r_n N(n)$ and $\lim \frac{r_{n+1}}{r_n}$ exist. Then, by Lemma \ref{lemma on limits},
\begin{align*}
    &\liminf_{n\rightarrow \infty} \frac{\lambda(F_n)}{N(n)} \sum_{j=0}^{N(n)-1} \frac{d \omega^{(j)}}{d \lambda} (g) \\
    &\geq \left( \liminf _{n\rightarrow \infty}\frac{\lambda(F_n)}{\lambda(E_{n})} \right) \left( 1 - \left( \lim _{n\rightarrow \infty}\frac{r_{n+1}}{r_n} \right)\right) \left( \frac{1 - e^{-\lim_{n\rightarrow \infty} r_n N(n)}}{\lim_{n\rightarrow \infty} r_n N(n)} - e^{-\lim_{n\rightarrow \infty} r_n N(n)} \right).
\end{align*}
Note that the continuous function $f(x) = \frac{1 - e^{-x}}{x} - e^{-x}$ is positive and bounded on $ (0, \infty)$ with $\lim_{x \to 0} f(x) =\lim_{x \to \infty} f(x)= 0$. %We observe that $\lim_{x \to 0} f(x) = 0$, $\lim_{x \to \infty} f(x) = 0$, $f(x) \neq 0$ for any $x \in (0, \infty)$, and $f(1)>0$. This implies that $f(x) > 0 $ for all $x \in (0, \infty)$.
Hence, one concludes that $\frac{1 - e^{-\lim r_n N(n)}}{\lim r_n N(n)} - e^{-\lim r_n N(n)}$ is  positive when  $\lim _{n\rightarrow \infty}r_n N(n)\in (0, +\infty)$.  %However,  the limit is $0$ if $\lim_{n\rightarrow \infty} r_n N(n)$ is $0$ or $\infty$. 
So, now all we need is to ensure that $  \lim _{n\rightarrow \infty}\frac{r_{n+1}}{r_n} <1$. We observe that $\frac{r_{n+1}}{r_n} \in [0,1]$.  If $0< \lim _{n\rightarrow \infty} \frac{r_{n+1}}{r_n} <1$, then asymptotically, $r_n \sim c^{-n}$ for some $c > 1$, and thus $N(n)$ is asymptotically $A c^n$ for some $A>0$. $\lim_{n\rightarrow \infty} \frac{r_{n+1}}{r_n} = 0$ also works, but then $N(n)$ increases superexponentially, which is less optimal. %than the case where $N(n) \sim c^n$. So, $r_n \sim c^{-n}$ and $N(n) \sim c^n$ for some $c>1$ is optimal, and yields the desired result. For the choice of $N(n) = 2^n$ and $t_n = 2^{-n}$ (which is equivalent to $r(n) = 2^{-(n-1)}$), one evaluates $\left( 1 - \left( \lim \frac{r_{n+1}}{r_n} \right)\right) \left( \frac{1 - e^{-\lim r_n N(n)}}{\lim r_n N(n)} - e^{-\lim r_n N(n)} \right) = \frac{1}{4}\left( 1 - \frac{3}{e^2} \right)$, which yields the particular estimate.

\end{proof}

\begin{rem}
    A standard choice of $t_n$ and $N(n)$ is $t_n = \frac{1}{2^n}$ and $N(n) = 2^n$. In this case, one can check that  %for all $g \in F_n$, 
    $$\liminf_{n\rightarrow \infty} \frac{\lambda(F_n)}{2^n}\sum_{j=0}^{2^n-1} \frac{d \omega^{(j)}}{d \lambda} %(g)
    \geq \frac{1}{4}\left(1 - \frac{3}{e^2} \right) \left( \liminf_{n\rightarrow \infty} \frac{\lambda(F_n)}{\lambda(E_{n})} \right) \chi_{F_n}.$$
\end{rem}

%\textbf{\textit{EDITS PENDING}}
%Now, %%if we have $F_n \subset F_{n+1}$ for all $n \in \mathbb{N}$, and 
%consider
%\begin{align*}
%    g \in &\left( \cap_{h_1, \dots, h_{j-1} \in E_{n'-1}} h_1 \dots h_{j-1} E_{n'} \right)
 %   \bigcap \left( \cap_{h_1, \dots, h_{j-1} \in E_{n'-1}} h_1 \dots h_{j-2} E_{n'} h_{j-1} \right)
  %  \bigcap \dots \\
   % &\bigcap \left( \cap_{h_1, \dots, h_{j-1} \in E_{n'-1}} h_1 E_{n'} h_2 \dots h_{j-1} \right) 
   % \bigcap \left( \cap_{h_1, \dots, h_{j-1} \in E_{n'-1}} E_{n'} h_1 \dots h_{j-1} \right).
%\end{align*}
%Then, for such $g$, by \cref{convolution absorption for n}, each of the $j$ convolution terms in the lower bound of $\frac{d \omega^{(j)}}{d \lambda}$ above equals $1$, and hence we get, 

%A weaker condition for $g$, albeit easier to state, is that $\left( E_{n'-1}^{j-1} \right)^{-1} g \left( E_{n'-1}^{j-1} \right)^{-1} \subset E_{n'}$. %\textcolor{blue}{Recall that we would need $F_n$ to be symmetric for all $n$, that is, $F_n = F_n^{-1}$, so that is equivalent to $F_{n'-1}^{j-1} g F_{n'-1}^{j-1} \subset F_{n'}$. Note, however, that if $G$ is abelian, then it is sufficient to consider $F_{n'-1}^{j-1} g \subset F_{n'}$.}

%If $G$ is amenable, then we can choose $G_n$ to be a sequence of F{\o}lner sets \cite{Fo55} such that labelling $F_n = F_{n-1}^{N(n)-2} G_n F_{n-1}^{N(n)-2}$, $\frac{\lambda\left(F_n \triangle G_n \right)}{\lambda\left(G_n\right)} \to 0$, that is, $\frac{\lambda\left(G_n\right)}{\lambda\left(F_n\right)} \to 1$. We elaborate on this a little bit below. 
We would like to note that no amenability assumptions on the group have been used in this subsection so far. To complete the proof of Theorem \ref{action dominance}, it remains to  construct compact sets $E_n$ and $F_n$ in $G$ such that $$\liminf_{n \to \infty}\frac{\lambda(F_n)}{ \lambda(E_n)} > 0.$$ In the following, we show that when 
$G$ is an amenable unimodular group, such compact sets indeed exist and, moreover, the limit equals 
$1$.

\begin{proof}[Proof of Theorem \ref{action dominance}]
 As discussed earlier, the existence of two-sided  F{\o}lner sequences of $G$ is guaranteed by Proposition \ref{biinvariant Folner sets existence}, and 
    without loss of generality we may also assume our F{\o}lner sets to be symmetric 
    and containing the group identity. Given such a F{\o}lner sequence $\{ F_n \}_{n \in \mathbb{N}}$, the goal now is to construct a subsequence $\setof{F_{n_k}}_{k\in\mathbb{N}}$ that satisfies, given $E_1 = F_{n_1} = F_1$ and $E_k = \left(E_{k-1}^{N(k)-2}\right)^{-1} F_{n_k} \left(E_{k-1}^{N(k)-2}\right)^{-1}$ for $k \geq 2$, $\lim_{k\rightarrow \infty} \frac{\lambda(F_{n_k})}{\lambda(E_{k})} = 1$.
    Then applying Proposition \ref{measure comparison in terms of sizes} to $F_{n_k}$ and $E_{n_k}$ immediately yields the result.
We choose a decreasing sequence $\setof{ \varepsilon_k }_{k \in \mathbb{N}}$ such that $\lim_{k\rightarrow \infty} \varepsilon_k = 0$. Since $\{ F_n \}_{n \in \mathbb{N}}$ is a two-sided  F{\o}lner sequence, there must exist $n_2\in \mathbb{N}$ such that $$\frac{\lambda\left(\left(E_1^{N(2)-2)}\right)^{-1} F_{n_2} \left(E_1^{N(2)-2)}\right)^{-1} \backslash F_{n_2} \right)}{\lambda( F_{n_2} )} < \varepsilon_2,$$
     and we set $E_2=\left(E_1^{N(2)-2)}\right)^{-1} F_{n_2} \left(E_1^{N(2)-2)}\right)^{-1}$.
     We proceed in a similar fashion for all $k \geq 2$ and obtain $F_{n_k}$ as a subsequence, where we choose $n_k$ such that  $$\frac{\lambda\left(E_k \backslash F_{n_k} \right)}{\lambda\left(F_{n_k}\right)} 
    < \varepsilon_k$$
    with $E_k = \left(E_{k-1}^{N(k)-2}\right)^{-1} F_{n_k} \left(E_{k-1}^{N(k)-2}\right)^{-1}$ for $k \geq 2$. 
    Consequently we get $\frac{\lambda(F_{n_k})}{\lambda(E_{k}) } \geq \frac{1}{1 + \varepsilon_k}$. Therefore, $\lim_{k\rightarrow \infty} \frac{\lambda(F_{n_k})}{\lambda(E_{k}) } = 1$ which completes the proof.
\end{proof}

%{\begin{lem}\label{lower bound for Fn/En size}
   % Let $G$ be a locally compact, second countable and unimodular amenable group. For any symmetric, two-sided , identity-containing F{\o}lner sequence $\{ F_n \}_{n \in \mathbb{N}}$ of $G$, % and \cref{balancing averaging sets}), 
    %there is a subsequence $\setof{F_{n_k}}_{k\in\mathbb{N}}$ that satisfies, given $E_1 = F_{n_1} = F_1$ and $E_k = \left(E_{k-1}^{N(k)-2}\right)^{-1} F_{n_k} \left(E_{k-1}^{N(k)-2}\right)^{-1}$ for $k \geq 2$, $$\lim_{k\rightarrow \infty} \frac{\lambda(F_{n_k})}{\lambda(E_{n_k})} = 1.$$
    %Moreover, every symmetric, identity-containing F{\o}lner sequence of $G$ has a subsequence that satisfies the invariance properties required for $F_n$ above.
%\end{lem}

\begin{rem}\label{potential obstructions for nonunimodular groups}
The unimodularity of the group is essential in our computation. This is mainly because, to obtain a positive lower bound for $\liminf _{n\rightarrow \infty}\frac{\lambda(F_n)}{N(n)} \sum_{j=0}^{N(n)-1} \frac{d \omega^{(j)}}{d \lambda} (g)$ for $g\!\in\! F_n$,
when estimating $\frac{d \omega^{(j)}}{d \lambda}$ from below, we must take into account all $n$ terms on the right-hand side of \eqref{eq: lower estimate}. To handle these terms, it is necessary to absorb uniform probability measures on both sides, which requires unimodularity, as indicated in Corollary \ref{convolution absorption for n}. To extend the comparison of measures to non-unimodular groups, one possible approach would be to find a way to handle the convolution products in \eqref{eq: conv sum}  whose highest index $n$ appears more than once. 

\end{rem}
\begin{rem}
For $T$ acting on the Hilbert space $L_2(\M)$, the adjoint $T^*$ is given by 
\begin{align*}
T^*(x) = \int _G\alpha_{g^{-1}}(x) d\omega(g) = \int _G\alpha_g(x) d\omega(g^{-1}).
\end{align*}
%The proof of this lemma follows from the properties of the Bochner integral as well as the fact that the action $\alpha$ preserves the trace $\tau$. 
%Indeed, for all $x,y$ in $L_2(\M)$, 
%$\langle T^*(x) , y \rangle = \langle x, T(y) \rangle 
 %   = \tau \left( x \left( T ( y \right)^* \right)
 %   = \tau \left( x \int \alpha_g(y^*) d\omega(g) \right)$.
%Now, by the elementary properties of the Bochner integral, one can equate this to %$\tau \left( \int x \alpha_g(y^*) d\omega(g) \right)$. Rewriting this as $\tau \left(\int \alpha_g \left( \alpha_{g^{-1}}(x) y^* \right) d\omega(g) \right) = \int \tau \left(\alpha_{g} \left( \alpha_{g^{-1}}(x) y^* \right) \right) d\omega(g)$, if we recall the fact that $\tau \circ \alpha_g = \tau$ for all $g \in G$, we can equate this to $\int \tau \left( \alpha_{g^{-1}}(x) y^* \right) d\omega(g) = \tau \left( \int \alpha_{g^{-1}}(x) y^* d\omega(g) \right)= \tau \left( \left( \int \alpha_{g^{-1}}(x) d\omega(g) \right) y^* \right)= \left\langle \left(\int \alpha_{g^{-1}}(x) d\omega(g) \right), y \right\rangle$.Since $x$ and $y$ are arbitrary, the claim follows. 
We recall the following definition from \cite{JX07}: $T$ is \textit{symmetric} relative to $\tau$, if $\tau \left( x T(y)^* \right) = \tau \left( T(x) y^* \right)$ for all $x,y \in L_2(\M)\cap \M$. And, we say $\omega$ is \textit{symmetric} if for all measurable subsets $E \subset G$, $\omega(E^{-1}) = \omega(E)$. It is direct to see that $T$ is symmetric iff $\omega$ is symmetric. %Observe that $T$ is symmetric iff $T=T^*$, which is true iff $\omega(E)=\omega(E^{-1})$ for every measurable subset $E \subset G$, that is, $\omega$ is symmetric. 
Our $\omega$ is indeed symmetric. In \cite[Theorem 6.7]{JX07} it was shown that if $T$ is symmetric and positive on $L_2(\M)$,  $T^n$ goes to $0$  b.a.u. or a.u. for $p$ in the suitable range as well. One can check that it is also possible to dominate the ergodic averages on $F_n$ by $T^n$. However, this does not yield any difference in the growth rate of the averaging sets.
\end{rem}

%\textcolor{red}{\textbf{PERHAPS IT IS POSSIBLE TO HAVE A UNIFORMLY BOUNDED $\varepsilon$ EVEN FOR SOME NONAMENABLE GROUPS? THAT WOULD BE SUFFICIENT FOR OUR THEOREM. SHOULD WE REMARK BRIEFLY ON THAT?} PERHAPS CHECK SL2Z, ETC, CHECK IF THIS PASSES ONTO SUBGROUPS, AS THEN YOU CHECK IF F2 IS A SUBGROUP, F2 FAILS THIS.}

\section{Applications to Maximal Inequalities and Individual Ergodic Theorems}
Based on the inequality \eqref{Ergodic Operator Inequality} in Theorem \ref{action dominance}, the maximal ergodic theorem and the individual ergodic theorems can be directly deduced from their counterparts for integer actions.
\subsection{Maximal inequalities for unimodular amenable groups}
The maximal ergodic theorem for amenable groups was proved in \cite[Theorem 6.4]{CW22}, the techniques used there are quite involved. In the following, we give a much simpler proof of this result for the case of unimodular amenable groups as an application of Theorem \ref{action dominance}.
\begin{thm}\label{maximal ergodic theorem with F_n}
    %[Theorem 6.4, second part, \cite{CW22}]
    Let $\left( \M, \tau \right)$ be a von Neumann algebra with a s.f.n. trace $\tau$. 
    Let $G$ be a locally compact second countable unimodular group with a {$w^\star$-continuous} action $\alpha: G \actson \left( \M, \tau \right)$ that preserves the trace $\tau$. 
    For every two-sided  F{\o}lner sequence in $G$, there exists a subsequence $\setof{F_n}_{n \in \mathbb{N}}$ such that if we define the averaging operators 
    $$A_n(x) := \frac{1}{\lambda{(F_n)}} \int_{F_n}\alpha_g(x) d\lambda(g),$$
    for $n\in\mathbb{N}$, $x \in L_p(\M)$, $p \in [1, \infty)$, 
    the sequence of operators $\left(A_n\right)_{n \in\mathbb{N}}$ is of strong type $(p,p)$ for $p \in (1, \infty)$ and of weak type $(1,1)$.
    %, for every $p \in (1, \infty)$, and for every $x \in L_p^+(\M)$, there exists $a \in L_p^+(\M)$ such that for all $n \in \mathbb{N}$,
    %$$\frac{1}{\lambda{(F_n)}} \int_{F_n}\alpha_g(x) d\lambda(g) \leq a \textrm{ and } \norm{a}_p \lesssim \norm{x}_p.$$
    %Moreover, given any two-sided  F{\o}lner sequence, one may obtain a subsequence which satisfies the properties required of $F_n$.
\end{thm}

\begin{proof}
    Recall that Theorem \ref{action dominance} implies for all $x \in L_p^+(\M)$, $0 \leq A_n(x) \leq C M_{N(n)}(x)$, where $M_{N(n)}(x)$  is the ergodic average for $T$ given in the form in \eqref{T} with $\omega$ being the measure given in \eqref{eq: def of w}.
    We first prove the strong type $(p,p)$ for $p \in (1, \infty)$.
    We obtain from Proposition \ref{JungeXuMaximalIntegers} and  \eqref{eq: charac for sup norm}, that for all $x \in L_p^+(\M)$, %$p \in (1, \infty)$, 
    there exists $a \in L_p^+(\M)$ such that $M_{N(n)}(x) \leq a$, and $\norm{a}_p \leq C_p \norm{x}_p$. Hence, $A_n(x) \leq Ca$, and consequently, $\norm{\sup_n^+ A_n(x)}_p \leq CC_p\norm{x}_p$.  
    Since every element of $L_p$ can be written as a linear combination of at most $4$ positive elements, we have $\norm{\sup_n^+ A_n(x)}_p \leq 4CC_p\norm{x}_p$ for all $x \in L_p(\M)$. 
    Yeadon's maximal inequality Proposition \ref{YeadonMaximalIntegers} says that for all $x \in L_1^+(\M)$ and $\e >0$, there exists a projection $e \in \M$ such that $\tau(1_\M - e) \leq \frac{4}{\e} \norm{x}_1$ and $eM_{N(n)}(x)e \leq \e e$ for all $n \in \mathbb{N}$. 
    Therefore, $A_n(x) \leq C M_{N(n)}(x)$ implies $eA_n(x)e \leq C eM_{N(n)}(x)e \leq \e C  e$. 
    Hence, $\left(A_n\right)_{n \in\mathbb{N}}$ is of weak type $(1,1)$. 
\end{proof}
%The proof of this is a direct application of \cref{action dominance} and \cref{JungeXuMaximalIntegers}. Using Yeadon's individual ergodic theorem for integers \cite{Ye77}, one may also obtain a weak type (1,1) inequality for $L_1$. 

To deduce the b.a.u. convergence for $p\in [1, 2)$ and the a.u. convergence for $p \in [2, \infty)$ from the maximal inequality, we will need the following noncommutative analogue of the Banach principle given in \cite{CL21}.

\begin{lem}[Theorem 3.1, \cite{CL21}]\label{nc banach principle}
Let $1\leq p <2$ (resp. $2 \leq p < \infty$) and $(S_n)_{n\geq 0}$ be a
sequence of positive linear maps on $L_p(\M)$ of weak type $(p,p)$. Then the space of the elements
$x\in L_p(\M)$ such that $S_n(x)$ converges b.a.u. (resp. a.u.) is closed in $L_p(\M)$.    
\end{lem}
We denote the fixed point subspace of the  group action $\alpha$ on $L_p(\M)$ by $\F^\alpha_p$, i.e.
$\F^\alpha_p:=\{x\in L_p(\M)\mid \alpha_g(x)=x \,\text{ for all } g\in G\}$. Consider a dense subset of $\F^\alpha_p$ given by $\mathcal{S}:={\rm{Span}}\{x-\alpha_g(x) \mid g\in G, x\in L_p(\M)\cap \M\}$. It is straightforward to check  that 
 $A_n (x)$ converges to $0$ a.u. for any $x\in \mathcal{S}$ and any $1\leq p<\infty$ (see, for instance, the proof of  \cite[Proposition 6.4]{HLW21}). Applying the Lemma above, we obtain the following individual ergodic theorem for unimodular amenable groups. 
\begin{thm}\label{individual ergodic theorem with F_n}
    %Let $\left( \M, \tau \right)$ be a semifinite von Neumann algebra $\M$ with a semifinite normal faithful trace $\tau$. 
    Let $G$ be a locally compact second countable {unimodular} group with a {$w^\star$-continuous} action $\alpha: G \actson \left( \M, \tau \right)$ that preserves the trace $\tau$. For every two-sided  F{\o}lner sequence of $G$, there exists a subsequence 
    %There exists a F{\o}lner sequence of compact subsets 
    $\setof{F_n}_{n \in \mathbb{N}}$ %of $G$, 
    such that, for every $p \in [1, \infty)$, and for every $x \in L_p(\M)$,
    $$A_n(x)=\frac{1}{\lambda{(F_n)}} \int_{F_n} \alpha_g(x) d\lambda(g) \longrightarrow \mathcal{P}(x) \text{ b.a.u.},$$ 
    where $\mathcal{P}$ is the projection from $L_p(\M)$ onto $\F^\alpha_p$. 
    For $p \!\in\! [2, \infty)$, $A_n(x)$ converges to $\mathcal{P}(x)$ a.u.. 
    %Moreover, given any F{\o}lner sequence, there exists a subsequence which satisfies the properties required of $F_n$.
\end{thm}

%Thus we have reproduced the maximal ergodic theorem 2.12 of \cite{CW22}, with the restriction of unimodularity. One may deduce b.a.u. convergence for $p\in [1, 2)$ and a.u. convergence for $p \in [2, \infty)$ from the maximal inequality by standard techniques, but we proceed to demonstrate direct proofs circumventing the maximal inequality in the following subsection.

%{Individual ergodic theorems follow directly from the maximal inequality for amenable groups \cref{maximal ergodic theorem with F_n}. That is a standard technique as demonstrated in \cite{CW22}. 
%However, in this section, we are going to provide the proof of the individual ergodic theorem without using a maximal inequality for amenable groups, directly from our operator inequality \ref{Ergodic Operator Inequality} and the individual ergodic theorem for integers \cref{JungeXuIntegers}.

\subsection{A direct proof of Individual Ergodic Theorems} \label{subsection: direct proof}
In this subsection, we will give a proof to Theorem \ref{individual ergodic theorem with F_n} without going through the maximal inequality for amenable groups. We will show the individual Ergodic Theorems for unimodular amenable groups directly from that for $\mathbb{Z}$. Oseledts used martingale techniques to show in \cite{Os65} that for actions on probability spaces, iterates of the Markov operator converge to the conditional expectation onto the subspace of invariant functions. We prove similar results in a more general setting here using basic functional analytic tools.

Recall that $\alpha: G \curvearrowright \left( \M, \tau \right)$ is a {$w^\star$-continuous} trace-preserving action on a von Neumann algebra with a s.f.n. trace. 
The Markov operator on $\M$ associated to $\alpha$ and a probability measure $\omega$ is defined as $T(x)= \int_G \alpha_g(x) d\omega(g)$ for all $x \in\M$. For simplicity, we also denote its extension to $L_p(\M)$ by $T$. 
Recall that $\F^T_p:=\{x\in L_p(\M)\mid T(x)=x \}$ and $\F^\alpha_p:=\{x\in L_p(\M)\mid \alpha_g(x)=x \,\text{ for all } g\in G\}$. Now we show that these two spaces coincide by showing that the $\omega$ defined in \eqref{eq: def of w} is a nondegenerate probability measure. 
%These results are straightforward and perhaps familiar to some, but we include them nonetheless for the sake of completeness. 

\begin{lem}\label{projectionsagree}
    If the support of $\omega$ generates $G$ as a group, then we have $\F_p^T \!=\!\F_p^\alpha$ for all $p\! \in \! [1, \infty)$.
\end{lem}
\begin{proof}
     Clearly $\F^\alpha_p \subset \F_p^T$ for all $p$; it remains to show the converse inclusion.   %(actually, one only needs 2-positivity for this \cite{Pa78}, but we have complete positivity of $T$ in any case). 
     Let $x \in \F^T_p$. Then we have $x^* \in \F^T_p $. Note that we have $$T(x^* x) - x^* x = T(x^* x) -T(x^*) x - x^* T(x) + x^* x = \int_G \abs{ \alpha_g(x) - x } ^2 d\omega(g)\geq 0.$$
Since $\tau \circ T \! = \! \tau$, we have that $\tau \!\left( T(x^* x) \!-\! x^* x \right) \! = \! 0$.
   Since $\tau$ is faithful, this implies $T(x^* x) - x^* x = 0$.
    This, in turn, implies $\int_G \abs{ \alpha_g(x) - x } ^2 d\omega(g) = 0$ and so $\alpha_g(x) - x = 0$ for almost every $g \in \supp(\omega)$. Since $\supp(\omega)$ generates $G$ as a group, by a standard continuity argument, we conclude that $\alpha_g(x)=x$ for every $g \in G$, thus $x \in \F^\alpha_p$. 
    %\textcolor{blue}{Given $x\in \F_p^T$, we define a function 
    %$f_x : G \to L_p(\M)$ by $f(g) = \alpha_g(x)-x$. 
   % $f_x$ is continuous. %(in $w^\star$ and hence in norm topology). 
    %Suppose there exists $g\in G$ such that $f_x(g) \neq 0$. 
    %By continuity of $f_x$, there exists some open neighbourhood $U_g$ of $g$ such that $f_x(g') \neq 0$ for any $g' \in U_g$. 
    %But, $\lambda(U_g)>0$ by the properties of the Haar measure, since $U_g$ is open. 
   % This is a contradiction to what we have proved above.}
\end{proof}

\begin{lem}\label{Full Support}
    The group generated by the support of $\omega$ defined in \eqref{eq: def of w} is $G$.
\end{lem}
\begin{proof}
   By the definition of $\omega$, we have  $\supp(\omega) \supset\cup_{n\in \mathbb{N}} E_n$. 
   By the definition of $E_n$, it is easy to see that $F_n\subset E_n$ for every $n\in \mathbb{N}$. 
   So it suffices to show that the group generated by $\cup_{n\in \mathbb{N}} F_n$, denoted by $\langle \cup_{n\in \mathbb{N}} F_n\rangle$ is indeed $G$. 
   Suppose that there exists a nonempty compact set $E\subset G$, 
   such that $E\cap \langle \cup_{n\in \mathbb{N}} F_n\rangle=\emptyset$. 
 We claim that for every $n \in \mathbb{N}$, 
   $EF_n \cap F_n = \emptyset$.
   Indeed, if $EF_n \cap F_n \neq \emptyset$, then 
 there exist 
   $g \in E$ and $h,k \in F_n$ such that $gh=k$, that is, $g = k h^{-1}$,
   which contradicts the assumption that $E\cap \langle \cup_{n\in \mathbb{N}} F_n\rangle=\emptyset$. 
   Thus $EF_n$ and $F_n$  are disjoint, and therefore
    $\frac{\lambda\left(E F_n \backslash F_n \right)}{\lambda\left( F_n \right)} 
   = \frac{\lambda\left(E F_n \right)}{\lambda\left( F_n \right)} 
   \geq 1$ for all $n \in \mathbb{N}$.
   This contradicts the F{\o}lner property of the sequence 
$\setof{F_n }_{n \in \mathbb{N}}$.
\end{proof}
This shows that the fixed point subspaces of $\alpha$ and $T$ agree for our choice of $\omega$. We shall identify $\F_p^T$ with $\F_p^\alpha$  henceforth, and denote them by $\F_p$. Therefore, to show the individual ergodic theorem for $G$, it suffices to work with $x\in \mathcal{F}_p^\perp$.

%Finally, we may see that for any $x \in \F_p$,
%$\frac{1}{\lambda(F_n)}\int_{F_n} \alpha_g(x) d\lambda(g) = x = \frac{1}{N(n)}\sum_{j=0}^{N(n)-1}T^j(x)$, for any averaging set $F_n$, $\lambda$ being a Haar measure on $G$. 
%We want to show that for $x$ in $\F_p^\perp$, $\lim \frac{1}{\lambda(F_n)}\int_{F_n} \alpha_g(x) d\lambda(g) = 0$ (a.u. / b.a.u.). We already have by \cref{JungeXuIntegers} that for such $x$, $\lim \frac{1}{n}\sum_{j=0}^{n-1} T^j(x) = 0$ (a.u. / b.a.u.). We are now in a position to embark on a proof of the individual ergodic theorem.}

\begin{proof}[A direct proof of Theorem \ref{individual ergodic theorem with F_n}]
     For any positive $x\in L_p(\M)$, let us denote $$M_{N(n)}(x) := \frac{1}{N(n)}\sum_{j=0}^{N(n)-1} T^j(x).$$
    Theorem \ref{action dominance} gives us that there exists a constant $C >0$ such that
    \begin{align*}
        0 \leq A_n(x) \leq C M_{N(n)}(x).
    \end{align*}
    By Lemmas \ref{projectionsagree} and \ref{Full Support}, it suffices to show that $\lim_{n\rightarrow \infty} A_n(x)= 0$  b.a.u. (resp. a.u.) for any positive $x$ in $\F_p^\perp$ when $p\in [1, \infty)$ (resp. $p\in [2, \infty)$). 
    For such $x$, 
    by the individual ergodic theorem for integer actions for $p \in [1, \infty)$ (proved in \cite{Ye77} for $p=1$ and in \cite[Corollary 6.4]{JX07} for $p \in (1, \infty)$), 
    for any $\varepsilon > 0$, there exists a projection $e \in \P(\M)$ such that $\tau(1_\M - e) \leq \varepsilon$, and $\lim_{n \to \infty} \norm{ e M_{N(n)}(x) e }_\infty = 0$. Multiplying the above inequality by $e$ on both sides, we get $0 \leq e A_n(x) e \leq C e M_{N(n)}(x) e$, and therefore $0 \leq \norm{e A_n(x) e}_\infty \leq C \norm{e M_{N(n)}(x) e}_\infty$. It is then easy to see that $\lim_{n \to \infty} \norm{e A_n(x) e}_\infty = 0$.

    If $p\geq2$, and hence $\frac{p}{2}\geq1$, then we have for any  $x \in \F_p^\perp$, $|x|^2 \in  L_{\frac{p}{2}}$. 
    By the b.a.u. convergence for the case $1\leq p<\infty$, for any $\varepsilon>0$ there exists a projection $e \in \P(\M)$ such that $\tau(1_\M - e) \leq \varepsilon$, and $\lim_{n \to\infty} \norm{ e A_n(|x|^2) e }_\infty = 0$. Applying the Kadison inequality %{\color{red}This operator inequalty does not generalise to Lp spaces... We have to assume x is in M, and then go through the nc Banach principal. But in this case, the proof is redundant, since we already know $A_n (x)$ has a.u. convergence in a dense subset.} 
    we have
    %$$\left| A_n(x) \right|^2 \leq A_n(|x|^2).$$
   % Multiplying the inequality above by $e$ on both sides, we get $e \left| A_n(x) \right|^2 e \leq e A_n(|x|^2) e$, which gives $$\left(A_n(x) e \right)^* \left(A_n(x) e \right) \leq \left( \left( A_n(|x|^2) \right)^{\frac{1}{2}} e \right)^* \left( \left( A_n(|x|^2) \right)^{\frac{1}{2}} e \right).$$
   % By the Douglas lemma, there exists a contraction $\c$ satisfying $\left(A_n(x) e \right)^* = \left( \left( A_n(|x|^2) \right)^{\frac{1}{2}} e \right)^* \c$, which yields $A_n(x) e = \c^* \left( A_n(|x|^2) \right)^{\frac{1}{2}} e $. Now we equate the norms of both sides and use the $C^*$-identity:
    $$\norm{ A_n(x) e }^2_\infty 
    = \norm{ e|A_n(x)|^2e }_\infty
    \leq  \norm{ e  A_n(|x|^2)  e }_\infty.$$
    The last inequality is easy to see if $x\in \M$; for general $x\in \F_p^\perp$, the inequality follows from a standard spectral truncation argument. Therefore $A_n(x)$ converges to $0$ a.u. for $p \in [2, \infty)$ when $x\in  \F_p^\perp$, which concludes the theorem. 
\end{proof}
\begin{rem}
    The argument in the second paragraph of the proof above is a general one, which shows that the implication -- b.a.u. convergence on $L_p$-spaces for $1\leq p<\infty$ implies a.u. convergence for $p\leq 2$ -- does not require the maximal inequality or the Banach principle. 
    %One may also note that this is somewhat analogous to a sandwich principle for convergence of sequences in the relevant topology. 
\end{rem}

\section{Ergodic Theorems in the Nontracial Setting} 

One may further push our results to noncommutative $L_p$-spaces in the nontracial setting. We will use the notion developed by Haagerup in \cite{Ha79a}. Throughout this section, $\M$ will be a von Neumann algebra equipped with a faithful normal state $\varphi$. 
Consider the antilinear operator $S_\varphi(x)=x^*$ on $L_2(\M, \varphi)$, its polar decomposition is written as $S_\varphi=J_\varphi\Delta_\varphi$. Recall that $\Delta_\varphi$ is the (unbounded)  modular operator such that $\Delta _\varphi^{it}\M\Delta _\varphi^{-it}=\M$ for all $t\in \R$. 
The modular automorphism group associated with $(\M, \varphi)$ is the one parameter automorphism group of $\M$ defined by 
$\sigma_t^\varphi (x)=\Delta _\varphi^{it}x\Delta_\varphi^{-it}$ 
for $x\in \M$ and $t\in \R$. 
Denote the corresponding crossed product by $\N:=\M \rtimes _\sigma\R$. 
There is a dual action $\hat \sigma_t$ of $\R$ on $\N$ such that $\hat \sigma_t(x)=x$ for $x\in \M$ and 
$\hat \sigma_t(\Delta_\varphi^{is})=e^{-ist}\Delta_\varphi^{is}$ for any $s\in \R$. 
On this crossed product, there is a canonical normal semi-finite faithful trace $\tau$ such that $\tau \circ \hat \sigma_t=e^{-t}\tau$. As a normal positive functional on $\M$, $\varphi$ corresponds to a positive element in $\M_*$ denoted by $D$, such that $$\varphi(x) = \tau (Dx)= \tau (xD)\; \textrm{ for all }\; x \in \M_*.$$
For any $p\in [1,\infty]$, an operator $x$ affiliated with  $\M \rtimes _\sigma\R$ is in $L_p(\M, \varphi)$ (the Haagerup noncommutative $L_p$-space) if and only if 
$$\hat \sigma_t (x)=e^{-\frac{t}{p}}x, \;\text{ and } 
\|x\|_p=\tau(|x|^p)^{\frac{1}{p}}, \text{ when } p \text{ is finite}.$$ 
The $L_p$-spaces satisfy the odd relation that $L_p(\M, \varphi) \cap L_q(\M, \varphi)=\{0\}$ if $p\neq q$,
which is a phenomenon that is different from the tracial case. However, $\M$ can be embedded in $L_p(\M,\varphi)$ through
\begin{align*}
 \iota_p:  \M_+ & \rightarrow L_p^+(\M,\varphi) \\
 x & \mapsto D^{\frac{1}{2p}}xD^{\frac{1}{2p}}.
\end{align*}
The symmetric embedding above can also be replaced by an asymmetric embedding, i.e. replacing $D^{\frac{1}{2p}}xD^{\frac{1}{2p}}$  by $D^{\frac{1-\theta}{p}}xD^{\frac{\theta}{p}}$ for any $\theta\in [0,1]$, which we denote by $\iota_{p,\theta}$.
The range of $\iota_{p,\theta}$ is dense in $L_p^+(\M,\varphi)$ (see \cite[Theorem 1.7]{GL95}). The embedding $\iota_{p,\theta}$  extends to $\M$ and we still denote it by $\iota_{p,\theta}$.
The definition of $L_p(\M;\ell_\infty)$ and the characterisation in \eqref{eq: charac for sup norm} extend verbatim to the present setting (see \cite{Ju02} and \cite{JX07}).

Junge and Xu extended their results Proposition \ref{JungeXuMaximalIntegers} and Corollary \ref{JungeXuIntegers} to the nontracial case as well in \cite{JX07}. We begin by stating the following conditions for a linear map $T:\M \to \M$:
 \begin{enumerate}
    \item[(T1')] \label{Tnontracialprops1} $T$ is a contraction on $\M$;
    \item[(T2')] \label{Tnontracialprops2} $T$ is completely positive;
    \item[(T3')]  \label{Tnontracialprops3} $\varphi \circ T \leq \varphi$;
    \item[(T4')]  \label{Tnontracialprops4} $T \circ \sigma_t^\varphi = \sigma_t^\varphi \circ T$ for all $t \in \R$.
\end{enumerate}
It was shown in \cite[Lemma 7.1]{JX07} that for every $T$ satisfying (T1'), (T2'), (T3') and (T4'), the map defined from the dense subspace $D^{\frac{1-\theta}{p}}\M D^{\frac{\theta}{p}}$ of $L_p(\M,\varphi)$ to itself via $\iota_{p,\theta}(x) \mapsto \iota_{p,\theta}(Tx)$ extends to a positive contraction of $L_p(\M,\varphi)$ for $p \in [1, \infty]$ (the extension is independent of $\theta$), which we shall denote by $T$ as well. 
Furthermore, they proved the following maximal inequality:
\begin{lem}[Theorem 7.4 part (i), \cite{JX07}]\label{JungeXuMaximalIntegersNontracial}
     If  $T$ satisfies  (T1'), (T2'), (T3') and (T4'), then for any $p \in (1, \infty)$ and any $x \in L_p(\M)$, 
    $$\norm{\sup_n{}^+ M_{n}(x)}_p \leq C_p \norm{x}_p$$
    where $C_p$ is the same constant as in Proposition \ref{JungeXuMaximalIntegers}.
\end{lem}

Now we consider  a locally compact second countable unimodular group $G$  and  an action $\alpha: G \actson \left( \M, \varphi \right)$ 
which satisfies the following conditions:
\begin{enumerate}
    %\item[(A1)]  For every $g\in G$, $\alpha_g$ is an automorphism on $\M$;
     \item[(A1)] For every $g\in G$, $\alpha_g \circ \varphi =\varphi$;
    \item[(A2)] For every $x\in \M$ and every $p\in [1,\infty)$,  the maps $g\mapsto \alpha_g(x)$ and $g\mapsto \iota_p\circ\alpha_g(x)$ are continuous with respect to the $w^*$-topology.
    \item[(A3)]   $\alpha$ commutes with the action of the modular group, that is, for every $g\in G$, $\alpha_g \circ \sigma_t^\varphi = \sigma_t^\varphi \circ \alpha_g$.
\end{enumerate}
Condition (A1) implies this action extends to $L_p(\M,\varphi)$ via $\iota _p(x) \mapsto \iota_p(\alpha_g (x))$ for any $g\in G$.
If we define $T$ as in \eqref{T}, that is, $T(x) = \int_G \alpha_g(x) d\omega(g)$, (A3) implies that condition (T4') holds for $T$.
 Conditions (T1'), (T2') and (T3') are satisfied directly following the same arguments as for (T1), (T2) and (T3).
Moreover, note that the ergodic average dominance \eqref{Ergodic Operator Inequality} in Theorem \ref{action dominance} still holds when the tracial von Neumann algebra is replaced by a nontracial one, since the proof ultimately reduces to \eqref{Measure Comparison Inequality}, which involves only the group $G$. Then applying Lemma \ref{JungeXuMaximalIntegersNontracial}, we may obtain the following maximal inequality for actions by any locally compact second countable unimodular amenable group on type III von Neumann algebras.

\begin{thm}\label{thm: maximal for non-tracial case}
    Let $\left( \M, \varphi \right)$ be a von Neumann algebra $\M$ with a  normal faithful state $\varphi$. Let $G$ be a locally compact second countable unimodular amenable group with an action $\alpha: G \actson \left( \M, \varphi \right)$ satisfying conditions (A1), (A2) and (A3). For every two-sided  F{\o}lner sequence in $G$, there exists a subsequence $\setof{F_n}_{n \in \mathbb{N}}$ such that if we define the averaging operators 
    $$A_n(x) := \frac{1}{\lambda{(F_n)}} \int_{F_n}\alpha_g(x) d\lambda(g),$$
    for $n\in\mathbb{N}$, $x \in L_p(\M,\varphi)$, $p \in [1, \infty)$, 
    the sequence of operators $\left(A_n\right)_{n \in\mathbb{N}}$ is of strong type $(p,p)$ for $p \in (1, \infty)$ and of weak type $(1,1)$.
\end{thm}

    The proof of the theorem above is identical to that of Theorem \ref{maximal ergodic theorem with F_n} and follows immediately from Lemma \ref{JungeXuMaximalIntegersNontracial} and \eqref{eq: charac for sup norm}.
    From this, one may obtain individual ergodic theorems with the suitable notion of convergence. Note that for $x_n$, $x\in L_p(\M,\varphi)$ with $p<\infty$, $x_n-x$ is not affiliated with $\M$ but with $\M \rtimes _\sigma\R$, the previous definition of a.u. and b.a.u. is not suitable for the current setting. Instead, we will need the almost sure convergence introduced by Jajte \cite{Ja91}.
     The sequence $(x_n)$ is said to converge bilaterally almost surely (b.a.s.) to $x$ if for every $\varepsilon>0$ there is a projection $e\in \M$ and a family $(a_{n,k})\subset \M$ such that
    $$
    \varphi (e^\perp)<\varepsilon, \; x_n-x=
    \sum_{k\geq 1} D^{\frac{1}{2p}} a_{n,k} D^{\frac{1}{2p}}\; \text{ and }\;\lim_{n\rightarrow\infty}\| \sum_{k\geq 1} (e\,a_{n,k} \,e)\|_{\infty}=0,
    $$
    where the first series converges in the norm of $L_p(\M,\varphi)$, the second in $\M$. The sequence $(x_n)$ is said to converge almost surely (a.s.) to $x$ if for every $\varepsilon>0$ there is a projection $e\in \M$ and a family $(a_{n,k})\subset \M$ such that
    $$
    \varphi (e^\perp)<\varepsilon, \; x_n-x=
    \sum_{k\geq 1}  a_{n,k} D^{\frac{1}{p}}\; \text{ and }\;\lim_{n\rightarrow\infty}\| \sum_{k\geq 1} (\,a_{n,k} \,e)\|_{\infty}=0.
    $$
    Junge and Xu proved the following individual ergodic theorem in the nontracial setting:
    \begin{lem} [Theorem 7.12 part (i), $d=1$, \cite{JX07}]\label{JungeXuIndividualIntegersNontracial}
    If $T$ satisfies (T1'), (T2'), (T3') and (T4'), then for any $p \in (1, \infty)$ and any $x \in L_p(\M,\varphi)$, 
$$M_{n}(x) \longrightarrow \mathcal{P}(x) \text{ b.a.s.},$$ 
    where $\mathcal{P}$ is the bounded projection from $L_p(\M,\varphi)$ onto the subspace of $T$-invariant elements. 
    For $p \in [2, \infty)$, $M_n(x)$ converges to $\mathcal{P}(x)$ a.s.. 
\end{lem}
The b.a.s. convergence of $M_n(x)$ for $p=1$ is also true (see \cite[Theorem 2.2.12]{Ja91}). %\cite{Go81}, \cite{GG88}.
This leads us to individual ergodic theorems in the nontracial case, which we state next. For completeness, we will provide the proof, adapted from the direct argument for the individual ergodic theorem in Subsection \ref{subsection: direct proof}, without using the maximal inequality in Theorem \ref{thm: maximal for non-tracial case}.
\begin{thm}
   % Let $\left( \M, \varphi \right)$ be a von Neumann algebra $\M$ with a normal faithful state $\varphi$. Let $G$ be a locally compact second countable unimodular group with an action $\alpha: G \actson \left( \M, \tau \right)$ that preserves the state $\varphi$ and commutes with the action of the modular group. 
   Under the same assumptions as in the previous theorem, for every two-sided  F{\o}lner sequence in $G$, there exists a subsequence $\setof{F_n}_{n \in \mathbb{N}}$ of $G$, such that, for every $p \in [1, \infty)$, and for every $x \in L_p(\M,\varphi)$,
    $$A_n(x) := \frac{1}{\lambda{(F_n)}} \int_{F_n}\alpha_g(x) d\lambda(g) \longrightarrow \mathcal{P}(x) \text{ b.a.s.},$$
    where $\mathcal{P}$ is the projection from $L_p(\M, \varphi)$ onto $\F_p:=\{x \!\in\! L_p(\M, \varphi)\mid \alpha_g(x)\!=\!x \,\text{ for all } g\in G\}$.
   For $p \in [2, \infty)$, $A_n(x)$ converges to $\mathcal{P}(x)$ a.s.. 
\end{thm}

\begin{proof}
We still have the decomposition $L_p(\M,\varphi)=\F_p\oplus \F_p^\perp$. It suffices to prove the result for positive $x\in \F_p^\perp$ (Lemmas \ref{projectionsagree} and \ref{Full Support} are still valid when we consider type III von Neumann algebras).
By Theorem \ref{action dominance} there exists a constant $C >0$ such that
    \begin{align*}
        0 \leq A_n(x) \leq C M_{N(n)}(x).
    \end{align*}
By the density of $D^{\frac{1}{2p}}\M_+ D^{\frac{1}{2p}}$ in $L_p^+(\M,\varphi)$, there are $a_{n,k}$ and $b_{n,k}$ in $\M_+$ such that $A_n(x) = \sum_{k\geq 1} D^{\frac{1}{2p}} a_{n,k} D^{\frac{1}{2p}}$ and $M_n(x) = \sum_{k\geq 1} D^{\frac{1}{2p}} b_{n,k} D^{\frac{1}{2p}}$. Using the fact that $\iota_p$ preserves positivity, $a_{n,k}$ and $b_{n,k}$ can be chosen to satisfy $a_{n,k}\leq C b_{N(n),k}$ for every $n, k\in \mathbb{N}$. Then it is easy to see that $\lim_{n\rightarrow\infty}\| \sum_{k\geq 1} (e\,b_{n,k} \,e)\|_{\infty}=0$ implies  $\lim_{n\rightarrow\infty}\| \sum_{k\geq 1} (e\,a_{n,k} \,e)\|_{\infty}=0$ for any projection $e$ in $\M$. As a consequence, the b.a.s. convergence for $1\leq p<\infty$ then follows from Lemma \ref{JungeXuIndividualIntegersNontracial} and \cite[Theorem 2.2.12]{Ja91}. 

Now consider the case $p\geq 2$. Let $x\in \F_p^\perp$. By the density of $\M D^{\frac{1}{p}}$ in $L_p(\M,\varphi)$, there exist $x_{k}\in \M$ such that $x_m:= \sum _{k=1}^mx_{k}D^{\frac{1}{p}}$ converges to $x$ in the $L_p$-norm. 
Applying the H\"older inequality for Haagerup $L_p$-spaces, one can show that $|x_m|^2$ converges to $|x|^2$ in the $L_{\frac{p}{2}}$-norm. By the fact that $A_n$ is a contraction on $L_q(\M,\varphi)$ for any $q\in [1,\infty)$, it is easy to see that 
$$
A_n(x_m)=A_n(\sum _{k=1}^mx_{k})D^{\frac{1}{p}} \longrightarrow A_n(x)
$$
in the $L_{p}$-norm, and 
$$A_n(|x_m|^2)=D^{\frac{1}{p}} A_n(|\sum _{k=1}^mx_{k}|^2)D^{\frac{1}{p}} \longrightarrow A_n(|x|^2)$$
in the $L_{\frac{p}{2}}$-norm. By the b.a.s. convergence for the case $1\leq p<\infty$, there exists a projection $e\in \M$ such that 
$$
\lim_{m,n\rightarrow\infty}\| eA_n(|\sum _{k=1}^mx_{k}|^2)e\|_\infty=0.
$$
Applying the Kadison inequality for $A_n$ on $\M$, we get $
\lim_{m,n\rightarrow\infty}\| e|A_n(\sum _{k=1}^mx_{k})|^2e\|_\infty= \lim_{m,n\rightarrow\infty}\| A_n(\sum _{k=1}^mx_{k})e\|^2_\infty=0$, which implies the a.s. convergence for $A_n(x)$. 

\end{proof}

Note that the above theorem can also be deduced from Theorem \ref{thm: maximal for non-tracial case} together with the noncommutative Banach principle Lemma \ref{nc banach principle}, adapted to the nontracial case (which remains valid).

\section{Explicit Computations of Averaging Sets}

In this section, we present explicit examples of two-sided F{\o}lner sequences which yield the ergodic average dominance \eqref{Ergodic Operator Inequality} in Theorem \ref{action dominance}. We focus on two classes of groups: groups of polynomial growth, and a prototype example of amenable groups with exponential growth -- the lamplighter group. Since the individual ergodic theorem follows from Theorem \ref{action dominance}, it is natural to compare these F{\o}lner sequences with those satisfying Lindenstrauss’s temperedness condition, which identifies the correct averaging sets for pointwise convergence in amenable group actions.
Recall that a left F{\o}lner sequence $(F_n)_{n\in \mathbb{N}}$ is called tempered if there exists a constant $C>0$ such that $\lambda\left(\cup_{i<n}F_i^{-1}F_n\right)\leq C \lambda\left(F_n\right)$, with the analogous condition for right-F{\o}lner sequences.
The F{\o}lner sequences used to obtain \eqref{Ergodic Operator Inequality} are tempered: each $F_n$ is symmetric and is approximately invariant from both sides under $\left( E_n^{-1} \right)^{N(n)-2}\supset E_n^{-1}$, and satisfies $E_n = E_n^{-1} \supset F_j$ for all $j<n$. Hence $F_n$ is approximately invariant from either side under $\cup_{i<n}F_i^{-1}$, which is precisely Lindenstrauss' temperedness condition.

    %Lindenstrauss’ averaging sets \cite{Li01} are required to satisfy the following criterion: there exists a constant $C>0$ such that $\lambda\left(\cup_{i<n}F_i^{-1}F_n\right)\leq C \lambda\left(F_n\right)$ for left-invariant F{\o}lner sets, or correspondingly, $\tilde{\lambda}\left(F_n \cup_{i<n}F_i^{-1}\right)\leq C \tilde{\lambda}\left(F_n\right)$ for right-invariant F{\o}lner sets.  That is,  %$\left(\cup_{i<n}F_i^{-1} F_n\right)\backslash F_n$ 
    %the boundary of $F_n$ remains small or uniformly controlled under translations by $\cup_{i<n}F_i^{-1}$ from one side. Our condition is more involved --indeed, iterative-- and strictly weaker, and it applies only in the unimodular case. This section is devoted to comparing estimates for averaging sets across different classes of groups. In groups of polynomial growth, our estimate is substantially weaker than Lindenstrauss’, whereas in the exponential-growth setting considered here, the two become more comparable.

\subsection{Groups of Polynomial Growth}   
For a group $G$ with polynomial growth of order $d$, generated by a symmetric compact set $B$, our averaging sets $F_n$ grow super-exponentially. 
%Our averaging sets grow really, really fast compared to some other examples currently extant in the literature. 
%, while the averaging sets for \cite{CW22} are of the form $F_n = B^{2^n}$ (radius grows exponentially \textcolor{red}{CHECK THIS!!!}). 
  %  For our case, the radius grows super-exponentially: a good candidate is $F_n = B^{2^{n^2}}$. In particular, anything less than superexponential growth can be seen to be insufficient for our case.
Let  $F_n = B^{l(n)}$ where $l: \mathbb{N} \to \mathbb{N}$ is a strictly increasing function to be determined. Then, $E_1 = F_1 = B^{l(1)}$, and one may compute using \eqref{eq: def of En} with $N(n) = 2^n$, that %$E_2 = B^{l(2) + 4 l(1)}$, $E_3 = B^{l(3)+ 12 l(2) + 48 l(1)}$, and one may prove that 
$E_n = B^{m(n)}$ with $m(n) = l(n) + 2\left( 2^n - 2 \right) m(n-1)$. 
Hence, we have $$\frac{\lambda(F_n)}{\lambda(E_n)} \sim \left(\frac{l(n)}{m(n)}\right)^d = \left(\frac{1}{1 + 2 \left( 2^n - 2 \right) \frac{m(n-1)}{l(n)}}\right)^d.$$
    For this to be strictly positive in the limit, %one must have $2 \left( 2^n - 2 \right) \frac{m(n-1)}{l(n)}$ to be uniformly bounded. We may note at this point that the order (i.e., polynomial, exponential, etc) of the function $m(n)$ must be not less than the order of the function $l(n)$. It is direct to see that
    polynomial or exponential form for $l(n)$ is not good enough, we need to consider at least super-exponential growth. One may verify that $$m(n) = l(n) + \sum_{j=1}^{n-1} l(j) \prod_{k=j}^{n-1} 2 \left( 2^{k+1} - 2 \right) = l(n) + \sum_{j=1}^{n-1} l(j) 2^{\frac{n^2 - j^2 + 3n -3j -4}{2}}\prod_{k=j}^{n-1}\left( 1 - \frac{1}{2^k} \right).$$
    With this, we have 
    $2 \! \left( 2^n - 2 \right) \! \frac{m(n-1)}{l(n)} 
    = 2(2^n - 2)\!\left( \! \frac{l(n-1)}{l(n)} + \sum_{j=1}^{n-2} \frac{l(j)}{l(n)}  2^{\frac{n^2 - j^2 + n -3j -6}{2}}\prod_{k=j}^{n-2}\left( 1 - \frac{1}{2^k} \right) \! \right)$. For this to be uniformly bounded, we need $\frac{l(n-1)}{l(n)} \sim 2^{-n}$ or less. So, a good candidate is $l(n)= 2^{n^2}$. With this choice of $l(n)$, we get
    \begin{eqnarray*}
   2 \left( 2^n - 2 \right) \frac{m(n-1)}{l(n)}& = &
     2(2^n - 2)\left( 2^{-2n + 1} + \sum_{j=1}^{n-2} 2^{\frac{-n^2 + j^2 + n -3j -6}{2}}\prod_{k=j}^{n-2}\left( 1 - \frac{1}{2^k} \right) \right)  \\
         & \leq & 2(2^n - 2)\left( 2^{-2n + 1} + \sum_{j=1}^{n-2} 2^{\frac{-n^2 + j^2 + n -3j -6}{2}} \right)\\
         & \leq &2(2^n - 2)\left( 2^{-2n + 1} + \left( n-2 \right) 2^{-3n+2} \right) ,
         %\longrightarrow 0 \textrm{ for } n \to \infty
         %\; \text{ which is uniformly bounded for all }n\in \mathbb{N}.
    \end{eqnarray*}
    which goes to $0$ for $n \to \infty$. 
    Hence, we have $\lim_{n\rightarrow\infty}\frac{\lambda(F_n)}{\lambda(E_n)} \sim 1$. Therefore, $F_n=B^{2^{n^2}}$'s are good averaging sets to give the ergodic average dominance \eqref{Ergodic Operator Inequality}. %From the computations above one can also see that $2^{n^2}$ is the optimal order for our invariance condition on the F{\o}lner sequence. 
     
     Now we consider individual ergodic theorems. 
     It is known that $F_n=B^{n}$'s are already enough to  satisfy Lindenstrauss' temperedness condition (observe that $F_n=B^n$ yields $\lambda\left(\cup_{i<n}F_i^{-1}F_n\right) = \lambda(F_{n-1} F_n) = \lambda(B^{2n-1}) \lesssim 2^d \lambda(F_n)$), while our chosen sequence $F_n$ grows much faster than needed in this setting.

     %Now, the continued product can be seen to be bounded above by $1$, which bounds the term above by 
   % $2(2^n - 2)\left( 2^{-2n + 1} + \sum_{j=1}^{n-2} 2^{\frac{-n^2 + j^2 + n -3j -6}{2}} \right)$. %\textcolor{blue}{It can be seen that this approximation does not change the order of $l(n)$.} Now we focus on the summand: we observe that the summand is eventually an increasing function of $j$, and since we are summing up to $n-2$, we may bound each summand above by $2^{\frac{-n^2 + (n-2)^2 + n - 3(n-2) - 6}{2}} = 2^{-3n+2}$. Hence, we have $2 \left( 2^n - 2 \right) \frac{m(n-1)}{l(n)} \leq 2(2^n - 2)\left( 2^{-2n + 1} + \left( n-2 \right) 2^{-3n+2} \right) = 2^{-n+2} - 2^{-2n+3} + 2^{-2n+3}n - 2^{-3n+4}n - 2^{-2n+4} + 2^{-3n+5}$, which goes to $0$ in the limit. Hence, we have
    %$$\lim\frac{\lambda(F_n)}{\lambda(E_n)} \sim 1.$$
    %This proves that $\omega$ chosen thus is a good candidate to set up $T$, and $F_n$'s are good averaging sets. %\textcolor{blue}
    %{Further investigation reveals that $2^{n^2}$ is the optimal order, for our invariance condition on F{\o}lner sets.

\subsection{The Lamplighter Group}

%We choose to study the lamplighter group since it is an amenable group of exponential growth, but its geometry is still nice enough that we may extract a lot of intuition about its F{\o}lner sets by looking at the Cayley graph. Not much is known in general about explicit F{\o}lner sequences for groups of exponential growth that satisfy the necessary conditions for ergodic averages to converge. Lindenstrauss has some restrictions on his averaging sets for the lamplighter group \cite{Li01}, he proves that they must grow superexponentially in size, but he does not provide an explicit form. We provide explicit forms in our paper.

Let us recall some of the basics of the lamplighter group: $\mathscr{L} = \Z\wr \Z_2$, the restricted wreath product of $\Z$ and the cyclic group $\Z_2$, modelling a set of countably many lamps on a bi-infinite street, and a lamplighter. 
%The lamplighter may take a step in either direction at a time, and may turn the lamp on or off at the position they are at. Initially all lamps are considered to be turned off, and we only allow finitely many operations, that is, at any given point of time, only finitely many lamps are turned on. We shall be working with the following presentation of $\mathscr{L}$, which allows for an easy characterisation of the F{\o}lner sets, as developed in \cite{St22}: 
It is the semi-direct product of $\Z\ltimes\left(\bigoplus_{\Z}\Z_2\right)$ with the 
     group operation defined as  $\left(t_1, K_1 \right) \cdot \left(t_2, K_2 \right) = \left(t_1 + t_2, K_1 \triangle(t_1 + K_2) \right)$. For any group element $\left(t, K \right)$,  $t\in \Z$ denotes the position of the lamplighter, and $K \subset \Z$ is a finite subset which denotes the set of lamps turned on. 
    It is easy to see that the group identity $e = (0, \emptyset)$, and $(t,K)^{-1}=\left( - t , - t + K \right)$.
   We shall use $\interval{n_1 , n_2}$ to denote $\left[n_1,n_2\right] \cap \Z$. All limits are for $n \to \infty$ unless specified otherwise. 

   A standard F{\o}lner sequence on $\mathscr{L}$ is given by 
   \begin{equation}\label{Lamplighter Right Folner}
    {\widetilde F}_n := \setof{\left(t, K \right) \mid t \in \interval{ 0 , n } \supset K} \; \text{ for any } n \in \mathbb{N}.   
   \end{equation} 
   It is well-known in the community that this produces a right-F{\o}lner sequence. For completeness, we provide a proof via computations which  set the tone for the rest of the subsection.
    \begin{lem}\label{lem: Lamplighter Right Folner}
        The sequence $\setof{{\widetilde F}_n}_{n \in \mathbb{N}}$ is a  right-F{\o}lner sequence of $\mathscr{L}$.
    \end{lem}
    \begin{proof}
        We first observe that $\abs{{\widetilde F}_n} = (n+1)2^{n+1}$. %where $n+1$ corresponds to the number of positions of the lamplighter (tracked by $t$), and $2^{n+1}$ corresponds to the possible number of configurations (lit/unlit) of the $n+1$ lamps (tracked by $K$), that is, all possible subsets of $\interval{0,n}$. 
       The aim is to prove that for all $t_1\in \Z$ and any finite subset $K_1 \subset \Z$, %$\lim \frac{\abs{ {\widetilde F}_n \left(t_1, K_1 \right) \backslash {\widetilde F}_n }}{\abs{ {\widetilde F}_n }} = 0$. %This is equivalent to proving that 
       $\lim \frac{\abs{{\widetilde F}_n \cap {\widetilde F}_n \left(t_1, K_1 \right) }}{\abs{ {\widetilde F}_n }} = 1$. 
       Note  that $${\widetilde F}_n \left(t_1, K_1 \right) = \setof{ \left(t + t_1, K \triangle \left(t+K_1\right) \right) \mid t \in \interval{ 0 , n } \supset K}.$$ 
       Since $K_1$ is finite, there exists $n_1$ such that $K_1 \subset \interval{-n_1 , n_1}$, 
       and if $\interval {-n_1 + t , n_1 + t} \subset \interval{0,n}$, 
       then $t + K_1 \subset \interval{0,n}$, 
       and $ K \triangle \left( t + K_1 \right) \subset \interval{0,n}$ as well. Hence, it is sufficient to show that for all $t_1\in \Z$ and all $n_1 \in \mathbb{N}$, $\lim \frac{\abs{{\widetilde F}_n\cap {\widetilde F}_n \left(t_1, \interval{ -n_1, n_1  } \right) }}{\abs{ {\widetilde F}_n }} = 1$. We may compute that $${\widetilde F}_n \left(t_1, \interval{ -n_1, n_1 } \right) = \setof{ \left(t + t_1, K \triangle \interval{ -n_1 + t , n_1 + t } \right) \mid t \in \interval{ 0 , n } \supset K 
        }.$$ 
       % The next step is to estimate which of these elements also belong to ${\widetilde F}_n$. 
       Now let $n\in \mathbb{N}$ such that $n> 2(|t_1|+n_1)$.  We may observe that given $t \in \interval{0,n}$ if $t_1 \geq 0$, then $t+t_1 \in \interval{0,n}$ iff $t \in \interval{0, n - t_1 }$; and if $t_1 < 0$, then $t+t_1 \in \interval{0,n}$ iff $t \in \interval{\abs{t_1}, n }$. Both these conditions are satisfied by $t \in \interval{\abs{t_1}, n - \abs{t_1}}$. And, for any $K \subset \interval{0,n}$, we have $K \triangle \interval{ -n_1 + t , n_1 + t } \subset \interval{0,n}$ iff $\interval{ -n_1 + t , n_1 + t } \subset \interval{0,n}$, which is equivalent to asking for $t \in \interval{ n_1 , n - n_1 }$. All of these conditions are satisfied by $t \in \interval{\abs{t_1} + n_1 , n - \abs{t_1} - n_1}$. That is, for all $t \in \interval{\abs{t_1} + n_1 , n - \abs{t_1} - n_1}$ and $K \subset \interval{0, n}$, $\left(t,K\right) \in {\widetilde F}_n \cap {\widetilde F}_n \left( t_1, \interval{-n_1,n_1}\right)$. 
        Hence, 
        $$\abs{{\widetilde F}_n \cap {\widetilde F}_n \left( t_1, \interval{-n_1,n_1}\right)} \geq \left( n + 1 - 2\left( \abs{t_1} + n_1 \right) \right) 2^{n+1} = \left( 1 - 2\frac{ \abs{t_1} + n_1 }{n+1} \right) \abs{{\widetilde F}_n}.$$
        The claim then follows by taking the limit as $n\rightarrow \infty$.
    \end{proof}

%By \cite{Li01}, there should exist a subsequence of the F{\o}lner sequence above that satisfies Lindenstrauss' invariance criterion. We find such a subsequence by a uniform averaging method. 
%The following approximation is a general technique for discrete groups, which we shall refer to henceforth as the \textit{uniform approximation technique}. It is very rough, but simplifies computations considerably, particularly for groups of exponential growth. 
    % \begin{enumerate}[label=(\roman*)]
    %    \item Let $\{ F_n \}_{n \in \mathbb{N}}$ be a left-invariant F{\o}lner sequence. Then, for any compact subset $K \subset G$, 
    %    $$\frac{\abs{KF_n\backslash F_n}}{\abs{F_n}} 
    %    \leq \sum_{g \in K}\frac{\abs{gF_n\backslash F_n}}{\abs{F_n}}
     %   \leq \sup_{g \in K}\frac{\abs{gF_n\backslash F_n}}{\abs{F_n}} \abs{K}.$$
        
     %   \item Similarly, let $\{ F_n \}_{n \in \mathbb{N}}$ be a right-invariant F{\o}lner sequence. Then, for any compact subset $K \subset G$, 
      %  $$\frac{\abs{F_nK\backslash F_n}}{\abs{F_n}} 
      %  \leq \sum_{g \in K}\frac{\abs{F_n g\backslash F_n}}{\abs{F_n}}
       % \leq \sup_{g \in K}\frac{\abs{F_n g\backslash F_n}}{\abs{F_n}} \abs{K}.$$

       % \item And finally, let $\{ F_n \}_{n \in \mathbb{N}}$ be a two-sided  F{\o}lner sequence. Then, for any compact subset $K \subset G$, 
      %  $$\frac{\abs{KF_nK\backslash F_n}}{\abs{F_n}} 
      %  \leq \sum_{g,h \in K}\frac{\abs{g F_n h \backslash F_n}}{\abs{F_n}}
       % \leq \sup_{g,h \in K}\frac{\abs{g F_n h\backslash F_n}}{\abs{F_n}} \abs{K}^2.$$
   % \end{enumerate}

%With this technique, we may obtain the averaging sets:
It was mentioned briefly in \cite{Li01} that the growth of tempered (right-)F{\o}lner sequences in lamplighter groups is super-exponential. We now provide a precise computation for the tempered subsequences of $\setof{{\widetilde F}_n}_{n \in \mathbb{N}}$.
\begin{prop}\label{Lindenstrauss averaging sets for the lamplighter group}
    If $\setof{{\widetilde F}_n}_{n \in \mathbb{N}}$ be the F{\o}lner sequence given in \eqref{Lamplighter Right Folner}, and $l:\mathbb{N} \to \mathbb{N}$ be a function defined as a sequence by $l(1)=1$ and $l(n)=3^{l(n-1)}$ for all $n \geq 2$, then $\setof{\widetilde F_{l(n)}}_{n \in \mathbb{N}}$ is a subsequence that is a tempered F{\o}lner sequence.
\end{prop}

\begin{proof}
    The estimate from the computation in Lemma \ref{lem: Lamplighter Right Folner} yields, for any $(t_1, K_1)$ such that $K_1 \subset \interval{-n_1 , n_1}$ we have 
    $\frac{\abs{{\widetilde F}_n \left( t_1, K_1\right) \backslash {\widetilde F}_n}}{\abs{{\widetilde F}_n}} \leq 2\frac{\abs{t_1} + n_1}{n+1}$. Hence, for any $(t_1, K_1)\in \widetilde F_{l(n-1)}$, we obtain that 
    \begin{eqnarray*}
   \frac{\abs{{\widetilde F}_{l(n)} {\widetilde F}_{l(n-1)} ^{-1} \backslash {\widetilde F}_{l(n)}}}{\abs{\widetilde F_{l(n)}} } &\leq & \frac{4l(n-1)}{l(n)+1}\abs{\widetilde F_{l(n-1)}} \\
        &= &\frac{4l(n-1)2^{l(n-1)+1}\left(l(n-1) + 1\right)}{l(n)+1} \\
        &\leq &\frac{16\left(l(n-1)\right)^2 2^{l(n-1)}}{l(n)}.   
    \end{eqnarray*}
%Recalling that $x^ab^x$ for $a>0$ and $b\in(0,1)$ is always bounded uniformly for positive $x$ (and in fact goes to $0$ in the limit $x \longrightarrow\infty$), 
One directly observes that the last estimate can be uniformly bounded iff $l(n)\sim\left(2+\eps\right)^{l(n-1)}$ for any chosen $\varepsilon>0$. To get integer values, the smallest choice is $\varepsilon=1$.
%This motivates the particular choice of $l(n) = 2^{2l(n-1)}$. In fact, one sees quite directly from the second line of the computation that this goes to zero in the limit, that is, defining $F_n = F^0_{l(n)}$, one has $\abs{F_n \cup_{i<n}{F_i}^{-1}} = \abs{F_n {F_{n-1}}^{-1}} \leq \abs{F_n}$
\end{proof}

However, note that the ${\widetilde F}_n$  defined in \eqref{Lamplighter Right Folner}  is not a left-F{\o}lner sequence and does not admit any subsequence that is a left-F{\o}lner sequence.
    For  $t_1>0$ and $n_1\in \mathbb{N}^*$, we have $$\left(t_1, \interval{-n_1,n_1}\right) {\widetilde F}_n = \setof{\left(t_1 + t , \interval{-n_1,n_1} \triangle \left(t_1 + K \right)\right) \mid t\in \interval{0,n} \supset K}.$$
    One can see that for any $K \subset \interval{0,n}$, $\interval{-n_1,n_1} \triangle \left(t_1 + K \right) = \interval{-n_1, -1} \sqcup \left(\interval{0, n_1} \triangle \left(t_1 + K \right)\right)$, which can never be a subset of $\interval{0,n}$. Hence, for any $(t_1, K_1)$ such that $K_1 \subset \interval{-n_1 , n_1}$, $\left( t_1 , K_1\right) {\widetilde F}_n\cap {\widetilde F}_n = \emptyset$ for any $n\in \mathbb{N}$. 
In the following, we will work with $$F_n ={\widetilde F}_n^{-1} {\widetilde F}_n,$$ 
and show that $F_n$ forms a two-sided F{\o}lner sequence of the lamplighter group. 

\begin{prop}\label{Lamplighter two-sided  Folner }\label{prop: biinvariant Folner in LL}
    The sequence $\setof{F_n}_{n \in \mathbb{N}}$ is a two-sided F{\o}lner sequence of $\mathscr{L}$.
\end{prop}
\begin{proof}
    We begin by computing \begin{align*}
        F_n = {{\widetilde F}_n}^{-1} {\widetilde F}_n & = \setof{\left(-t, -t + K \right)\left(t', K' \right) \mid t,t' \in \interval{0,n} \supset K,K'} \\
        &= \setof{\left(-t + t', \left(-t + K\right) \triangle \left(-t+K'\right) \right) \mid t,t' \in \interval{0,n} \supset K,K'} \\
        &= \setof{\left(-t + t', -t + K \triangle K'\right) \mid t,t' \in \interval{0,n} \supset K,K'} \\
        &= \setof{\left(-t + t', -t + K \right) \mid t,t' \in \interval{0,n} \supset K},\\
        &= \setof{ \left(-t, \emptyset \right) \left( t', K \right) \mid t,t' \in \interval{0,n} \supset K},
    \end{align*}    
    which implies that $F_n$ consists of elements of ${\widetilde F}_n$ shifted to the left by an integer in $\interval{0,n}$. A direct observation yields that  $\abs{F_n} \leq \left(n+1\right)^22^{n+1}$. %($n+1$ choices for $t$ and $t'$ each, and $2^{n+1}$ choices for $K$). 
    A lengthy but straightforward computation reveals that  
    $\abs{F_n} = 2^n\left(n^2 + 4n + 2 \right)$.
    We need to prove that for all $t_1, t_2 \in \Z$ and all finite subsets $K_1, K_2 \subset \Z$, 
    $\lim \frac{\abs{ F_n \cap \left( t_1 , K_1 \right) F_n \left( t_2 , K_2 \right) }}{\abs{F_n}} = 1$.
    By an argument similar to the one in the proof of Lemma \ref{lem: Lamplighter Right Folner}, we may replace $K_1$ by $\interval{-n_1,n_1}$ and $K_2$ by $\interval{-n_2,n_2}$ for some $n_1, n_2 \in \mathbb{N}$. Hence, it is sufficient to prove that for all $t_1, t_2 \in \Z$ and all $n_1, n_2 \in \mathbb{N}$, 
    $$\lim \frac{\abs{ F_n \cap \left( t_1 , \interval{-n_1 , n_1} \right) F_n \left( t_2 , \interval{-n_2 , n_2} \right) }}{\abs{F_n}} = 1 .$$
    We first compute 
    \begin{align*}
        &\left( t_1 , \interval{-n_1 , n_1} \right) F_n \left( t_2 , \interval{-n_2 , n_2} \right) \\
        &= \Bigl\{ \left( -t + t_1 +t' + t_2, \interval{-n_1,n_1} \triangle \left( -t +t_1 + K \right) \triangle \interval{ -t +t_1 +t' -n_2 , -t + t_1+t' +n_2} \right) \\
        & \;\;\;\;\;\;\mid t,t' \in \interval{0,n} \supset K\Bigr\} \\
        & = \Bigl\{\left( t_1 - t , \emptyset \right) \left( t' + t_2, \interval{t-t_1-n_1, t-t_1+n_1}\triangle K \triangle \interval{t'-n_2 , t'+n_2}\mid t,t' \in \interval{0,n} \right) \supset K \Bigr\}.
    \end{align*}
 We now attempt to estimate the intersection: $\abs{ g_1 F_n g_2 \cap F_n}$ where $g_1 = \left( t_1 , \interval{-n_1 , n_1} \right)$ and $g_2 = \left( t_2 , \interval{-n_2 , n_2} \right)$. %that is, we ask, which $t,t'$ and $K$ satisfy the overlap condition that $\left( - \left(t - t_2 \right) , \emptyset \right) \left( t' + t_1, \interval{t-t_2-n_2 , t-t_2+n_2} \triangle K \triangle \interval{t'-n_1, t'+n_1} \right) \in F^{00}_n$.
    Consider $n>2\max\{|t_1|+n_1, |t_2|+ n_2\}$, we 
    observe that if $t \in \interval{\abs{t_1} + n_1, n - \abs{t_1} - n_1}$ and $t' \in \interval{\abs{t_2} + n_2, n - \abs{t_2} - n_2}$, then, for any $K\subset \interval{0,n}$, $\left(-t, \emptyset \right) \left( t', K \right)\in g_1 F_n g_2 \cap F_n$. Therefore, this allows us to estimate 
    %We consider the contrapositive: if the overlap condition fails for some $K$, then either $t \in \interval{0,n}\backslash\interval{\abs{t_1} + n_1, n - \abs{t_1} - n_1}$ or $t' \in \interval{0,n}\backslash\interval{\abs{t_2} + n_2, n - \abs{t_2} - n_2}$ (or both!) must hold. This allows us to estimate 
    \begin{align*}
         \abs{g_1F_n g_2 \backslash F_n}
         &\leq 2\left(\abs{t_1}+n_1\right)\left(n+1\right)2^{n+1} + 2\left(\abs{t_2}+n_2\right)\left(n+1\right)2^{n+1} \\
        &= 2\left(\abs{t_1} + \abs{t_2} + n_1 + n_2 \right) \left(n+1\right)2^{n+1}.
    \end{align*}
   Consequently,
    \begin{align*}
        \frac{\abs{g_1F_n g_2 \backslash F_n}}{\abs{F_n}}
        &\leq \frac{2\left(\abs{t_1} + \abs{t_2} + n_1 + n_2 \right) \left(n+1\right)2^{n+1}}{2^n\left( n^2 + 4n + 2 \right)} \\
        &\leq \frac{4\left(\abs{t_1} + \abs{t_2} + n_1 + n_2 \right) }{\left( n + 1 \right)}.
    \end{align*}    
    Taking the limit $n \rightarrow \infty$, the claim is proved.
\end{proof}
 \begin{cor}\label{uniform estimate for two-sided  lamplighter Folner sets}
        Let $F_n$ and $F_m$ be two-sided F{\o}lner sets as constructed above, with $n>4m$. Then, 
        $$\frac{\abs{F_m F_n F_m \backslash F_n}}{\abs{F_n}} \leq \frac{144m^5 4^m}{n+1}.$$
    \end{cor}
 
    \begin{proof}
        The proof follows directly from the last estimate in the proof of Proposition \ref{prop: biinvariant Folner in LL}:
        \begin{equation*}
        \frac{\abs{F_m F_n F_m \backslash F_n}}{\abs{F_n}}\leq \frac{16m}{n+1}\abs{F_m}^2=  \frac{16m}{n+1}4^{m}\left(m^2 + 4m + 2 \right)^2 \leq \frac{144m^5 4^m}{n+1},    
        \end{equation*}
    which proves our estimate.    
    \end{proof}

 %which we shall refer to henceforth as the \textit{shape-preserving estimate}. 
 
  We are now ready to extract subsequences that serve as good averaging sets for establishing the ergodic average dominance \eqref{Ergodic Operator Inequality} for the lamplighter group.

    \begin{prop}\label{our averaging sets for the lamplighter group}
        Let $\setof{F_n}_{n \in \mathbb{N}}$ be the two-sided F{\o}lner sequence as above. If $l:\mathbb{N} \to \mathbb{N}$ is a function defined by $l(1)=1$ and $l(n)=17^{2^n l(n-1)}$ for all $n \geq 2$, then $\setof{F_{l(n)}}_{n \in \mathbb{N}}$ is a subsequence that gives ergodic average dominance \eqref{Ergodic Operator Inequality}.
    \end{prop}

    \begin{proof}
        %We shall be using the shape-preserving estimate to estimate $E_{n-1}$, and the uniform approximation estimate in the form of Lemma \ref{uniform estimate for two-sided  lamplighter Folner sets} for $E_n$. 
        One may observe from the definition of $F_n$ that $F_n = F_n^{-1}$, $(0, \emptyset)\in F_n$, and $ F_n F_m\subset F_{n+m}$ for any $m,n\in \mathbb{N}$.
        Therefore,  using \eqref{eq: def of En} with $N(n)=2^n$, and repeating the same computation as in the polynomial-growth case for balls, we obtain $E_{n-1} = E_{n-1}^{-1} \subset F_{m(n-1)}$, where 
        $m(n) 
        = l(n) + \sum_{j=1}^{n-1} l(j) 2^{\frac{n^2 - j^2 + 3n -3j -4}{2}}\prod_{k=j}^{n-1}\left( 1 - \frac{1}{2^k} \right)$. Hence, by Corollary \ref{uniform estimate for two-sided  lamplighter Folner sets}, we have 
        \begin{align*}
            \frac{\abs{E_n \backslash F_{l(n)}}}{\abs{F_{l(n)}}}
            &= \frac{\abs{\left(E_{n-1}^{-1}\right)^{2^n-2} F_{l(n)} \left(E_{n-1}^{-1}\right)^{2^n-2} \backslash F_{l(n)}}}{\abs{F_{l(n)}}} \\
            &\leq \frac{\abs{{\left(F_{m(n-1)}\right)}^{2^n-2} F_{l(n)} {\left(F_{m(n-1)}\right)}^{2^n-2} \backslash F_{l(n)}}}{\abs{F_{l(n)}}} \\
            &\leq \frac{\abs{{F_{\left(2^n-2\right)m(n-1)}} F_{l(n)} {F_{\left(2^n-2\right)m(n-1)}} \backslash F_{l(n)}}}{\abs{F_{l(n)}}} \\
            &\leq 144 \frac{\left(2^n-2\right)^5\left(m(n-1)\right)^5 4^{\left(2^n-2\right)m(n-1)}}{l(n)+1} \\
            %&\leq 144\frac{\left(2^{n} m(n-1)\right)^5 4^{2^n m(n-1)}}{l(n)+1} \\
            %&= 144\frac{\left(m(n-1)\right)^5 2^{2^{n+1} m(n-1) + 5n}}{l(n)+1} 
            &\leq 144\frac{2^{5n}\left(m(n-1)\right)^5 4^{2^n m(n-1)}}{l(n)+1},
    \end{align*}
    given $l(n)>4(2^n-2)m(n-1)$.
    Hence, a good candidate for $l(n)$ is $l(n) \gtrsim \alpha^{2^n m(n-1)}$ for any $\alpha>4$. %, recalling that $x^ab^x$ for $a>0$ and $b\in(0,1)$ is always bounded uniformly for positive $x$ (and in fact goes to $0$ in the limit $x \longrightarrow\infty$). 
    %Now, we must replace $m$ by $n$. This can be affected in the following fashion. 
    Observe that if we assume $l(n) = \alpha^{2^n l(n-1)}$, then for all $n \geq 2$,  %hence, 
    %for $n \geq 2$, % $2^n\left( l(n-1) - \frac{1}{2}l(n-2) \right) > 2n^2 \log(2)$, which implies $2^n l(n-1) - 2^{n-1}l(n-2) > \log(n-1) + n^2 \log(2)$ 
    \begin{align*}
        \log_\alpha\left(l(n)\right) - \log_\alpha\left(l(n-1) \right)
        &= 2^n l(n-1) - 2^{n-1} l(n-2) \\
        &= 2^n \left( l(n-1) - \frac{1}{2} l(n-2) \right) \\
        &> 2^{n} > \log_\alpha(n-1) + n^2 \log_\alpha(2).
    \end{align*}
    This implies that $\frac{l(n)}{l(n-1)} > 2^{n^2}(n-1)$. Using this estimate, one notes that
    \begin{align*}
        m(n) 
        &= l(n) + \sum_{j=1}^{n-1} l(j) 2^{\frac{n^2 - j^2 + 3n -3j -4}{2}}\prod_{k=j}^{n-1}\left( 1 - \frac{1}{2^k} \right) \\
        &< l(n) + \sum_{j=1}^{n-1} l(j) 2^{\frac{n^2 - j^2 + 3n -3j -4}{2}} \\
        %&< l(n) + 2^{n^2}\sum_{j=1}^{n-1} l(j) \\
        &< l(n) + 2^{n^2}\sum_{j=1}^{n-1} l(j) \\
        &< l(n) + 2^{n^2}(n-1)l(n-1)< 2l(n).
    \end{align*}
    Hence, going back to our initial choice, $l(n) \gtrsim \alpha^{2^n m(n-1)}$, %and observing that $m(n) < 2l(n)$, 
    a good choice would be $l(n) \sim \left(\alpha^2\right)^{2^n l(n-1)}$. To get integer values, the smallest choice that works is $l(n) = 17^{2^n l(n-1)}$.
    \end{proof}

\section*{Acknowledgements}
The authors would like to thank L{\'e}onard Cadilhac, Simeng Wang, Cameron Wilson, and Quanhua Xu for their valuable suggestions and discussions.

\section*{Declarations}
\textbf{Conflict of Interest: }
The authors have no conflicts of interest to disclose.

\textbf{Data availability: }
There is no data involved with this project.

%\printbibliography
\bibliographystyle{abbrvdin}
\bibliography{bibliography.bib}

\end{document}